\documentclass[12pt]{amsart}

\usepackage{mathtools}
\usepackage{amsmath}
\usepackage{graphicx}
\usepackage[all]{xy}
\usepackage{todonotes}
\usepackage{enumerate}
\usepackage{verbatim}
\usepackage{multirow}
\usepackage{pgf,tikz}
\usepackage{longtable}
\usepackage{wrapfig}

\usepackage{algcompatible}
\usepackage{algorithm}
\usepackage{algpseudocode}

\usepackage{bigstrut}
\usepackage[english]{babel}
\usepackage{hyperref}
\usepackage{amsmath,amssymb,amsthm,textcomp,mathrsfs,latexsym}
\usepackage{lmodern}
\usepackage{multicol}
\usepackage{mdwlist}
\usepackage{pdfpages}
\usepackage{array}
\usepackage{amscd}
\usepackage{fancyvrb}										
\usepackage{microtype}
\usepackage{fancyhdr}
\usepackage{setspace} 
\usepackage{tikz}
\usepackage[left=3cm,top=2.5cm,right=3cm]{geometry}
\usetikzlibrary{arrows}
\usetikzlibrary{matrix}
\usetikzlibrary{shapes.geometric}
\usetikzlibrary{calc}

\hypersetup{
    colorlinks=true,       
    linkcolor=blue,          
    citecolor=blue,        
    filecolor=blue,      
    urlcolor=blue           
}

\usepackage{color}
\definecolor{red-}{rgb}{1.0,0.0,0.0}
\definecolor{green-}{rgb}{0.0,0.4,0.0}
\definecolor{blue-}{rgb}{0.0,0.0,1.0}

\usepackage{hyperref}
\hypersetup{
    colorlinks=true,
    linkcolor=blue,
    citecolor=blue,
    filecolor=blue,
    urlcolor=blue
}

\theoremstyle{theorem}
\newtheorem{theorem}{Theorem}[section]
\newtheorem{corollary}[theorem]{Corollary}
\newtheorem{lemma}[theorem]{Lemma}
\newtheorem{proposition}[theorem]{Proposition}

\theoremstyle{definition}
\newtheorem{remark}[theorem]{Remark}

\newtheorem{definition}[theorem]{Definition}

\newtheorem{example}[theorem]{Example}

\begin{document}

\title[Skew partial derivatives and Hermite interpolation]
{On skew partial derivatives and a Hermite-type interpolation problem}

\author[J.A. Briones and A.L. Tironi]{Jonathan Armando Briones Donoso and Andrea Luigi Tironi}

\date{\today}

\address{
Universidad de Concepci\'on,
Facultad de Ciencias F\'isicas y Matem\'aticas,
Departamento de Matem\'atica,
Casilla 160-C,
Concepci\'on, Chile}
\email{atironi2006@gmail.com}
\email{jonathanbriones@udec.cl}

\subjclass[2021]{Primary: 14J70, 11G25; Secondary: 05B25. Key words and phrases: Skew polynomials, Hermite interpolation, Division ring.}

\begin{abstract}
Let $\mathcal{R}:=\mathbb{F}[{\bf x};\sigma,\delta]$ be a multivariate skew polynomial ring over a division ring $\mathbb{F}$. In this paper, we introduce the notion of right and left $(\sigma,\delta)$-partial derivatives of polynomials in $\mathcal{R}$ and we prove some of their main properties. As an application of these results, we solve in $\mathcal{R}$ a Hermite-type multivariate skew polynomial interpolation problem. The main technical tools and results used here are of constructive type, showing methods and algorithms to construct a polynomial in $\mathcal{R}$ which satisfies the above Hermite-type interpolation problem and its relative Lagrange-type version. 
\end{abstract}

\maketitle

\section*{Introduction}\label{intro}
In \cite{Ore}, O. Ore introduced univariate skew polynomial rings and principally developed the theory of these non-commutative polynomial rings over a division ring (skew field), also called in literature Ore extensions. After this work, many other authors studied the main properties of these rings and produced new results with them in many other areas, principally in abstract algebra, quantum groups (it has been observed that most of these special Hopf algebras can be expressed in the form of a skew polynomial ring), coding theory and cryptography (e.g. \cite{Jaco}, \cite{Jate1}, \cite{Maj}, \cite{BGU}, \cite{BGGRU}), extending the corresponding classical ones. 

Successively, an extension of these skew polynomial rings to the multivariate case was quite natural and a first intent was given by considering the so-called iterated skew polynomial rings (see e.g. \cite[$\S 8.8$]{CohnFree} and \cite{GU}). Unfortunately, unique remainder algorithms do not hold in general for iterated skew polynomials because they do not satisfy Jategaonkar's condition for many variables (\cite{CohnFree}, \cite{Jate2}, \cite{Len}).

Recently, in \cite{evaluationandinterpolation} the authors proposed an alternative construction by defining free multivariate skew polynomial rings and following Ore's idea, that is, the product of two monomials consists in appending them and the degree of a product of two skew polynomials is the
sum of their degrees. In particular, over fields and adding commutativity between constants and variables, the free multivariate skew polynomial rings yield the conventional free algebra as a special case. Moreover, they showed that with this new definition one can define the evaluation of any (free) skew polynomial $F:=F(x_1,\dots,x_n)$ at any point $(a_1,\dots,a_n)$ as the unique remainder of the Euclidean division of the polynomial $F$ by the skew polynomials $x_1-a_1, \dots, x_n-a_n$ on the right. Thanks to this, they were able to describe under some finiteness conditions the functions obtained by evaluating multivariate skew polynomials over division rings, a problem closely related to that of Lagrange-type interpolation which was studied previously in the univariate case in \cite{Er1, Lam, LL, LMK, Zhang}. 

\smallskip

Inspired principally by \cite{Er1}, and then by \cite{zerosmartinez}, the main purpose here is to introduce the notion of right and left $(\sigma,\delta)$-partial derivatives of multivariate skew polynomials and to solve a Hermite-type interpolation problem in the free multivariate skew polynomial rings, extending some previous results for the univariate case and for the skew Lagrange-type interpolation problem (Theorem \ref{theorem 3.1.7} and Corollary \ref{Tlagrange}).

\smallskip

The paper is organized as follows. In $\S 1$, we recall some basic notation and definitions, the main properties of the free multivariate skew polynomial ring $\mathbb{F}[{\bf{x}};\sigma,\delta]$ (Definition \ref{defR}) and we give some preliminary results. In $\S 2$, we introduce the notion of right and left $(\sigma,\delta)$-partial derivatives for free multivariate skew polynomials (Definitions \ref{Der1.1} and \ref{Der1.2.}) and we give some of their basic properties (Lemmas \ref{Der1.5} and \ref{Der1.7}). Finally, in $\S 3$, after some technical and algorithmic tools, we prove the main result of this paper, that is, a Hermite-type interpolation problem in $\mathbb{F}[{\bf{x}};\sigma,\delta]$ and its associated Lagrange-type interpolation problem as a special case. Finally, these results are specialized to the univariate case and skew Vandermonde matrices are considered to solve with a different approach the above interpolation problems in the multivariate case.

\section{Background material}\label{Sect1}

Let $\mathbb{F}$ be a division ring, that is, a unitary ring (not necessarily
commutative) in which every non-zero element is invertible in $\mathbb{F}$.  For positive integers $m$ and $n$, $\mathbb{F}^{m \times n}$ will denote the set of $m\times n$ matrices over $\mathbb{F}$. In particular,  $\mathbb{F}^{m\times 1}$ denotes the set of column vectors of length $m$ and $\mathbb{F}^{n}:=\mathbb{F}^{1\times n}$ the set of row vectors of length $n$ over $\mathbb{F}$.

\smallskip

Following \cite{evaluationandinterpolation}, we begin by recalling the free multivariate skew polynomial ring which corresponds to a multivariate generalization of the ring of skew polynomials defined by Ore in \cite{Ore}. To do this, we first need to recall the concept of $\sigma$-vector derivation.

\begin{definition}\label{def 1.2.1}{\normalfont \cite[Definition 1]{evaluationandinterpolation}} 
Given a ring homomorphism $$\sigma:\mathbb{F}\to \mathbb{F}^{n \times n},\;a \mapsto \left(
\begin{matrix}
\sigma_{1,1}(a) & \sigma_{1,2}(a) & \cdots & \sigma_{1,n}(a)\\
\sigma_{2,1}(a) & \sigma_{2,2}(a) & \cdots & \sigma_{2,n}(a) \\
\vdots & \vdots & \ddots & \vdots \\
\sigma_{n,1}(a) & \sigma_{n,2}(a) & \cdots & \sigma_{n,n}(a) 
\end{matrix}
\right), $$ we say that
$$\delta: \mathbb{F} \to \mathbb{F}^{n \times 1}\;,\;a \mapsto \left(
\begin{matrix}
\delta_{1}(a)\\
\delta_{2}(a)\\
\vdots\\
\delta_{n}(a) 
\end{matrix}
\right)$$
is a $\sigma$-\textit{vector derivation} (over $\mathbb{F}$) if it is an additive group homomorphism and satisfies $\delta(ab)=\sigma(a)\delta(b)+\delta(a)b$ for all $a, b \in \mathbb{F}$, where the maps  $\sigma_{i,j}:\mathbb{F}\to \mathbb{F}$ and $\delta_i:\mathbb{F}\to \mathbb{F}$ for $i,j\in\{1,\dots,n \}$ are the component functions of $\sigma$ and $\delta$, respectively. The ring
homomorphism $\sigma : \mathbb{F}\to \mathbb{F}^{n\times n}$
will be called a \textit{matrix morphism}.
\end{definition}

\begin{example}\label{ex 1.2.2}{\normalfont \cite[Example 4]{evaluationandinterpolation}}
Let $\sigma : \mathbb{F}\to \mathbb{F}^{n\times n}
$ be a matrix morphism and let ${\bf{v}} \in \mathbb{F}^{n\times 1}$. The map
$\delta_{{\bf{v}}}: \mathbb{F} \to \mathbb{F}^{n}
$ defined by
$\delta_{{\bf{v}}}(a) := \sigma(a){{\bf{v}}} - {{\bf{v}}} a$
for all $a \in \mathbb{F}$ is a $\sigma$-vector derivation called inner derivation. 
\end{example}

Let us give now the main basic definitions of this paper.

\begin{definition}\label{defR}
Let $x_1, x_2, \dots , x_n$ be $n$ pair-wise distinct letters, called \textit{variables},
and denote by $\mathcal{M}$ the set of all finite strings using these variables, that is, the free (non-commutative) monoid with (left) basis $x_1, x_2, \dots , x_n$. The empty string will be denoted by $o$, an element $m \in \mathcal{M}$ will be called a \textit{monomial} and its \textit{degree}, denoted by $\deg(m)$, is its length as a string. The \textit{free multivariate skew polynomial ring} over $\mathbb{F}$ in the variables $x_1, x_2, \dots , x_n$ with matrix morphism $\sigma$ and $\sigma$-vector derivation $\delta$, denoted by $$\mathcal{R}:=\mathbb{F}[\textbf{x}; \sigma, \delta]\ ,$$ 
is the free left $\mathbb{F}$-module with basis $\mathcal{M}$ and product given by appending monomials using the rule
\begin{eqnarray}
{}^t\textbf{x}\cdot a=\sigma(a){}^t\textbf{x}+\delta(a) \label{1.4}
\end{eqnarray}
for all $a \in \mathbb{F}$, where $\textbf{x}:=(
x_1,x_2,\cdots, x_n
)\in \mathcal{M}^{n}$ and ${}^t\textbf{x}$ is the transpose of $\textbf{x}$. Each
element $F(\textbf{x}) \in \mathcal{R}$ is called a \textit{free multivariate skew polynomial}, or simply skew polynomial, and can be expressed uniquely as a left linear combination 
$$F(\textbf{x})= \sum_{m \in \mathcal{M}} F_mm,$$
where $F_m \in \mathbb{F}$ are all zero except for a finite
number of monomials $m \in \mathcal{M}$.
\end{definition}

\begin{remark}
The expression \eqref{1.4} is a short form of writing $x_ia= \sum_{j=1}^{n} \sigma_{i,j}(a)x_j+\delta_i(a),$
for every $i = 1, \dots ,n$. 
\end{remark}

\begin{remark}
Consider the functions $\sigma : \mathbb{F}\to \mathbb{F}^{n\times n}
$ and $\delta : \mathbb{F}\to \mathbb{F}^{n\times 1}$ defined as $\sigma(a)=aI$ and $\delta (a)=\vec{0}$ for every $a\in\mathbb{F}$, respectively, where $I$ is the $n\times n$ identity matrix and $\vec{0}$ is the null column vector. In this case, we will write simply $\sigma = Id$ and $\delta =0$. Then, $\sigma$ is a matrix morphism, $\delta$ is a $\sigma$-vector derivation and the ring $\mathcal{R}=\mathbb{F}[\textbf{x}; Id, 0]$ is the free conventional polynomial ring in
the variables $x_1, \dots , x_n$ (see \cite[Sec. 0.11]{CohnFree}) which do not commute with each other, but commute with constants.
\end{remark}

Thanks to the lack of relations among the variables, it was proven in \cite[Lemma 5]{evaluationandinterpolation}
that for any $a_1,a_2, \dots ,a_n \in \mathbb{F}$ and every $F(\textbf{x})\in \mathcal{R}$ there exist unique $G_1(\textbf{x})$, $G_2(\textbf{x}), \dots ,$
$G_n(\textbf{x})\in \mathcal{R}$ and $b\in \mathbb{F}$ such that
\begin{equation}\label{for 1.6}
    F(\textbf{x})=\sum_{i=1}^{n}G_i(\textbf{x})(x_i-a_i)+b \ .
\end{equation}

\noindent The next result shows that one can obtain also a left-hand version of \cite[Lemma 5]{evaluationandinterpolation}.

\begin{lemma}\label{lema5bis}
Consider $a_1, \dots , a_n\in \mathbb{F}$. Then, for every {\em $F(\textbf{x}) \in \mathcal{R}$} there exist unique
{\em $G_{1,L}(\textbf{x}),\dots ,$ $G_{n,L}(\textbf{x}) \in \mathcal{R}$} and $b\in \mathbb{F}$ such that 
{\em
\begin{equation}\label{4}
    F(\textbf{x})=\sum_{i=1}^{n}(x_i-a_i)\cdot G_{i,L}(\textbf{x})+b
\end{equation}
}
if and only if the map $\varphi:\mathbb{F}^{n}\to \mathbb{F}^{n}$ defined by 
\begin{equation}\label{6}
    \varphi\left((\gamma_1,\gamma_2, \dots ,\gamma_n)\right):=\left(\sum_{j=1}^{n}\sigma_{j,1}(\gamma_j),\sum_{j=1}^{n}\sigma_{j,2}(\gamma_j), \dots ,\sum_{j=1}^{n}\sigma_{j,n}(\gamma_j)\right) 
\end{equation} 
is an isomorphism of additive groups, where $\sigma_{i,j}:\mathbb{F}\to \mathbb{F}$ are the \textit{component functions} of $\sigma:\mathbb{F}\to \mathbb{F}^{n\times n}$.
\end{lemma}

\begin{proof}
First, note that $\varphi$ is an additive group homomorphism because all the $\sigma_{i,j}$'s are additive group homomorphisms. Moreover, since $F(\textbf{x})$ is a sum of monomials, to prove the statement will be sufficient to consider only monomials of the form  $\alpha_{k}x_k$ $(\alpha_k \in \mathbb{F})$ for $k=1,2,..,n$. 

\noindent Let $\alpha:=(\alpha_1,\alpha_2, \dots ,\alpha_n)\in \mathbb{F}^{n}$ and write 
$$\alpha_kx_k=\sum_{j=1}^{n}(x_j-a_j)G_{j,k}+b_k\;\;\;(*)$$
for some $G_{j,k},b_k\in \mathbb{F}$. By \eqref{1.4}, for every $k=1, \dots ,n$, we have 
$$\alpha_k x_k=\sum_{i=1}^{n}\left(\sum_{j=1}^{n}\sigma_{j,i}(G_{j,k}) \right)x_i-\sum_{j=1}^{n}a_jG_{j,k}+\sum_{j=1}^{n}\delta_{j}(G_{j,k})+b_k\ ,$$
i.e. 
$$\alpha_k=\displaystyle{\sum_{j=1}^{n}\sigma_{j,k}(G_{j,k})}, b_k=\displaystyle{\left(\sum_{j=1}^{n}a_jG_{j,k}-\sum_{i=1}^{n}\delta_{j}(G_{j,k})\right)},\ \ \displaystyle{\sum_{j=1}^{n}\sigma_{j,i}(G_{j,k})=0}\ \ \forall i\neq k.\;(**)$$
Thus, we get
$$
\begin{array}{rcl}
 \varphi((G_{1,1},G_{2,1}, \dots ,G_{n,1})) & = & (\alpha_1,0, \dots ,0)
  \\ \varphi((G_{1,2},G_{2,2}, \dots ,G_{n,2})) & = & (0,\alpha_2, \dots ,0)
  \\ \vdots & = & \vdots
  \\ \varphi((G_{1,n},G_{2,n}, \dots ,G_{n,n})) & = & (0,0, \dots ,\alpha_n) \ .
\end{array}
$$ 
Therefore, we have $\varphi((\beta_1,\beta_2, \dots ,\beta_n))=\alpha$ for $(\beta_1,\beta_2, \dots ,\beta_n)\in \mathbb{F}^{n}$ with $\beta_i:=\sum_{j=1}^{n}G_{i,j}$ for $i=1, \dots ,n$. This proves that $\varphi$ is a surjective homomorphism. On the other hand, by the uniqueness of the $G_{j,k}$'s for all $k=1, \dots ,n$, it follows that $\varphi$ is injective and therefore an isomorphism of additive groups.
Conversely, if $\varphi$ is an isomorphism of additive groups, then given $(\alpha_1,0, \dots ,0),(0,\alpha_2, \dots ,0), \dots ,(0, \dots ,0,\alpha_n)\in\mathbb{F}^{n}$ there exist unique $G_{1,k},G_{2,k}, \dots ,G_{n,k}, b_k\in \mathbb{F}$ for all $k=1, \dots ,n$ such that $(**)$ holds and $\alpha_k x_k$ can be written as in $(*)$ in a unique manner.
\end{proof}

\begin{remark}
If $\sigma=Id$, then $\varphi$ is the identity morphism, i.e $\varphi(\vec{\gamma})=\vec{\gamma}$ for any $\vec{\gamma}\in\mathbb{F}^n$. In particular, equation \eqref{4} always holds for any $F(\textbf{x})\in \mathbb{F}[\textbf{x};Id,0]$. 
\end{remark}

\begin{remark}
For $n=1$, the condition that $\varphi$ is a group isomorphism is equivalent to require that $\sigma$ is a (ring) automorphism of $\mathbb{F}$. On the other hand, in the special case when $\sigma:\mathbb{F}\to \mathbb{F}^{n\times n}$ is a diagonal homomorphism, that is, $\sigma_{i,j}(a)=0$ for all $i\neq j$, it follows that any skew polynomial in $\mathcal{R}$ can be written as in \eqref{4} if and only if the component functions $\sigma_{i,i}:\mathbb{F}\to \mathbb{F}$ are ring automorphisms. Moreover, in this situation we can give explicit formulas to write $ax_i$ as in \eqref{4}. Indeed, for any $i=1, \dots ,n$, we have $$ax_i=(x_i-a_i)\sigma_{i,i}^{-1}(a)+a_i\sigma_{i,i}^{-1}(a)-\delta_{i}(\sigma_{i,i}^{-1}(a)).$$
\end{remark}

\smallskip

\noindent {\bf Notation}. \textit{From now on, when we will need to use \eqref{4}, we will always assume that \eqref{6} holds, that is, $\varphi$ defined as in \eqref{6} is an isomorphism of additive groups.}

\begin{remark}\label{remark0}
With the same notation as in the proof of Lemma \ref{lema5bis}, for every $k=1,\dots,n$, we have
$$\alpha_kx_k=\sum_{j=1}^{n}(x_j-a_j)G_{j,k}+b_k=$$
$$=(\textbf{x}-\textbf{a})\ {}^t\varphi^{-1}(\alpha_k\vec{e}_k)+
\textbf{a}\ {}^t\varphi^{-1}(\alpha_k\vec{e}_k)-\psi\varphi^{-1}(\alpha_k\vec{e}_k)=\textbf{x}\ {}^t\varphi^{-1}(\alpha_k\vec{e}_k)-\psi\varphi^{-1}(\alpha_k\vec{e}_k)\ ,$$
i.e. $\alpha_kx_k=\textbf{x}\ {}^t\varphi^{-1}(\alpha_k\vec{e}_k)-\psi\varphi^{-1}(\alpha_k\vec{e}_k)$, 
where $\vec{e}_k\in\mathbb{F}^n$ is the $k$-th canonical vector and $\psi:\mathbb{F}^n\to\mathbb{F}$ is such that
$$
\psi (c_1,\dots ,c_n):= \sum_{i=1}^{n}\delta_i(c_i)\ \in\ \mathbb{F}\ .
$$
Moreover, if we define 
$$\hat{(f)}
\left( 
\begin{matrix}
{}^t\textbf{b}_1\ | \
\dots \ | \
{}^t\textbf{b}_n
\end{matrix}
\right):=\left(
\begin{matrix}
{}^tf(\textbf{b}_1)\ | \
\dots \ | \
{}^tf(\textbf{b}_n)
\end{matrix}\right)
$$
for any map $f:\mathbb{F}^n\to\mathbb{F}^h$ with $h\in\mathbb{Z}_{>0}$, we can rewrite for any $\lambda\in\mathbb{F}$ the above relations $\lambda x_k=\textbf{x}\ {}^t\varphi^{-1}(\lambda\vec{e}_k)-\psi\varphi^{-1}(\lambda\vec{e}_k)$, as follows:
$$\lambda\textbf{x}=\textbf{x}\ \hat{(\varphi^{-1})}(\lambda I)-\hat{\left(\psi\varphi^{-1}\right)}(\lambda I)\ ,$$ where $I$ is the $n\times n$ identity matrix. Therefore, if \eqref{6} holds, then for any $\lambda\in\mathbb{F}$ we can write
\begin{equation}\label{lambda x}
\lambda\textbf{x}=\textbf{x}\ \tilde{\sigma}(\lambda)+\tilde{\delta}(\lambda)\ ,
\end{equation}
where
\begin{equation}\label{tilde functions}
\tilde{\sigma}(\lambda):=\hat{(\varphi^{-1})}(\lambda I)\ , \ \ \tilde{\delta}(\lambda):=-\hat{\left(\psi\varphi^{-1}\right)}(\lambda I)\ .
\end{equation}
Moreover, since for every $\lambda,\mu\in\mathbb{F}$ we have
$$\textbf{x}\tilde{\sigma}(\mu\lambda)+\tilde{\delta}(\mu\lambda)=(\mu\lambda)\textbf{x}=\mu\textbf{x}\tilde{\sigma}(\lambda)+\mu\tilde{\delta}(\lambda)=\textbf{x}\tilde{\sigma}(\mu)\tilde{\sigma}(\lambda)+\tilde{\delta}(\mu)\tilde{\sigma}(\lambda)+\mu\tilde{\delta}(\lambda)\ ,$$
we conclude that
\begin{equation}\label{tilde properties}
\tilde{\sigma}(\mu\lambda)=\tilde{\sigma}(\mu)\tilde{\sigma}(\lambda)\ , \quad \tilde{\delta}(\mu\lambda)=\tilde{\delta}(\mu)\tilde{\sigma}(\lambda)+\mu\tilde{\delta}(\lambda) \ .
\end{equation}
Finally, if $\sigma = Id$, $\delta=0$, then we deduce that $\tilde{\sigma}=Id$, $\tilde{\delta}=0$ and $\lambda\textbf{x}=\textbf{x}\lambda$ for any $\lambda\in\mathbb{F}$.
\end{remark}

\bigskip

We can now define the right (left) $(\sigma, \delta)$-evaluation of every $F(\textbf{x})\in\mathcal{R}$ at any $\textbf{a}\in\mathbb{F}^{n}$ as follows (see also \cite[Definition 9]{evaluationandinterpolation}).

\begin{definition}\label{def-eval}
Take $\textbf{a} = (a_1, a_2,\cdots, a_n) \in \mathbb{F}^n$ and consider any skew polynomial $F(\textbf{x})\in \mathcal{R}$. We define the \textit{right (left) $(\sigma, \delta)$-evaluation of} $F(\textbf{x})\in\mathcal{R}$ \textit{at} $\textbf{a}$, denoted by $F(\textbf{a})$ ($F_L(\textbf{a})$), as the unique constant $b\in \mathbb{F}$ of \eqref{for 1.6} (of \eqref{4} if $\varphi$, defined as in \eqref{6}, is an isomorphism of additive groups). In particular, we denote by $N_{m}({\bf a})\ \left( M_m({\bf a})\right)$ the right (left) $(\sigma,\delta)$-evaluation of any $m\in\mathcal{M}$.
\end{definition}

\begin{remark}
Consider $\lambda,\mu\in\mathbb{F}$ and $F(\textbf{x}), G(\textbf{x}),\in\mathcal{R}$. For any $\textbf{a}\in\mathbb{F}^n$, from Definition \ref{def-eval} and the above properties, it follows that
$$\left( \lambda F(\textbf{x})+\mu G(\textbf{x})\right)(\textbf{a})=\lambda F(\textbf{a})+\mu G(\textbf{a})\ , \ \left( F(\textbf{x})\lambda+ G(\textbf{x})\mu \right)_L(\textbf{a})= F_L(\textbf{a})\lambda+ G_L(\textbf{a})\mu\ .$$ 
\end{remark}

\begin{example}
Consider $\mathbb{F}[x_1,x_2;Id,0]$ with $\mathbb{F}$ a division ring such that $ba\neq ab$ for some $a,b\in\mathbb{F}$. Take $F(x_1,x_2):=(x_1-a)bx_2\in\mathbb{F}[x_1,x_2;Id,0]$. Note that we can also write $F(x_1,x_2)$ as
$$F(x_1,x_2)= (bx_1-ab)(x_2-1)+b(x_1-a)+(ba-ab) \ .$$
This gives $F_L(a,1)=0$ and $F(a,1)=ba-ab\neq 0$, that is, $F_L(a,1)\neq F(a,1)$.
\end{example}

The following result allows one to compute the right (left) $(\sigma, \delta)$-evaluation of any monomial $m\in \mathcal{M}$ at any point $\textbf{a}\in \mathbb{F}^n$ by using induction with the so-called fundamentals functions $N_{m}:\mathbb{F}^n \to \mathbb{F}$ ($M_{m}:\mathbb{F}^n \to \mathbb{F}$) for any $m\in \mathcal{M}$. 

\begin{theorem}
Consider a monomial $m\in\mathcal{M}$ and a point ${\bf{a}}\in \mathbb{F}^{n}$. Then the following properties hold:
\begin{enumerate}
\item[$1)$] {\em \cite[Theorem 2]{evaluationandinterpolation}}
it holds that
$$N_{{\bf{x}}m}({\bf{a}}):=
\left( 
\begin{matrix}
N_{x_1m}({\bf{a}})\\
N_{x_2m}({\bf{a}})\\
\vdots\\
N_{x_nm}({\bf{a}})
\end{matrix}
\right)=\sigma(N_m({\bf{a}}))\ {}^t{\bf{a}}+\delta(N_m({\bf{a}}))\ ;
$$
\item[$2)$] 
using the same notation as in {\em Remark \ref{remark0}}, it holds that
$$M_{m{\bf{x}}}({\bf{a}}):=
\left( 
\begin{matrix}
M_{mx_1}({\bf{a}}), 
M_{mx_2}({\bf{a}}), \dots ,
M_{mx_n}({\bf{a}})
\end{matrix}
\right)={\bf{a}}\ \tilde{\sigma}(M_m({\bf{a}}))+\tilde{\delta}(M_m({\bf{a}}))\ .$$
\end{enumerate}
\end{theorem}

\begin{proof}
Assume we are in case $2)$. Write $m=\sum_{j=1}^{n}(x_j-a_j)G_j(\textbf{x})+M_{m}(\textbf{a})$. For every $i=1,\dots,n$, we have 
$$mx_i=\sum_{j=1}^{n}(x_j-a_j)G_j(\textbf{x})x_i+M_{m}(\textbf{a})x_i=$$
$$=\sum_{j=1}^{n}(x_j-a_j)G_j(\textbf{x})x_i+\textbf{x}\ {}^t\varphi^{-1}(M_{m}(\textbf{a})\vec{e}_i)-\psi\varphi^{-1}(M_{m}(\textbf{a})\vec{e}_i)=$$
$$=\sum_{j=1}^{n}(x_j-a_j)G_j(\textbf{x})x_i+(\textbf{x}-\textbf{a})\ {}^t\varphi^{-1}(M_{m}(\textbf{a})\vec{e}_i)+\textbf{a}\ {}^t\varphi^{-1}(M_{m}(\textbf{a})\vec{e}_i)-\psi\varphi^{-1}(M_{m}(\textbf{a})\vec{e}_i)\ ,$$
i.e. $M_{mx_i}(\textbf{a})=\textbf{a}\ {}^t\varphi^{-1}(M_{m}(\textbf{a})\vec{e}_i)-\psi\varphi^{-1}(M_{m}(\textbf{a})\vec{e}_i)$. Therefore the statement $2)$ follows from Remark \ref{remark0}.
\end{proof}

Now, to define properly the evaluation of a product of two skew polynomials, we first need the following notion. 

\begin{definition}\label{def 1.2.10}
Consider $\textbf{a} \in \mathbb{F}^{n}$ and $c \in \mathbb{F}^{*}$. Then
\begin{enumerate}
\item[$1)$] \normalfont{\cite[Definition 11]{evaluationandinterpolation}} the \textit{right $(\sigma,\delta)$-conjugate of $\bf{a}$ with respect to $c$} is
$${\textbf{a}}^c := {}^t\left(\sigma(c)\ {}^t\textbf{a}c^{-1} + \delta(c)c^{-1}\right)\in \mathbb{F}^{n}\ ;$$
\item[$2)$] the \textit{left $(\sigma,\delta)$-conjugate of $\bf{a}$ with respect to $c$} is
$${}^c \textbf{a}:=c^{-1}\textbf{a}\tilde{\sigma}(c)+c^{-1}\tilde{\delta}(c)\in \mathbb{F}^{n}\ .$$
\end{enumerate}
\end{definition}

The following properties extend the well-known ones for the case $n = 1$. 

\begin{lemma}\label{lemgeneralizado}
Given ${\bf{a}},{\bf{b}} \in \mathbb{F}^{n}$ and $c,d \in \mathbb{F}^{*}$, the following properties hold:
\begin{itemize}
    \item[$1)$] \cite[Lemma 12]{evaluationandinterpolation} ${\bf{a}}^1={\bf{a}}$, $({\bf{a}}^{c})^{d}={\bf{a}}^{dc}$ and the relation $\sim$ defined on $\mathbb{F}^{n}$ as
    
\smallskip    

 \centerline{${\bf{a}} \sim {\bf{b}}$ $\iff\ \exists\ e \in \mathbb{F}^{*}$ such that ${\bf{b}}={\bf{a}}^{e}$,} 

\smallskip    

\noindent is an equivalence relation on $\mathbb{F}^{n}$;
\item[$2)$] ${}^1{\bf{a}}={\bf{a}}$, ${}^d({}^c {\bf{a}})={}^{cd}{\bf{a}}$ and the relation $\sim_L$ defined on $\mathbb{F}^{n}$ as
    
\smallskip    

 \centerline{${\bf{a}} \sim_L {\bf{b}}$ $\iff\ \exists\ e \in \mathbb{F}^{*}$ such that ${\bf{b}}={}^{e}\bf{a}$,} 

\smallskip    

\noindent is an equivalence relation on $\mathbb{F}^{n}$.
\end{itemize}
\end{lemma}

\begin{proof}
To prove $2)$, first note that $\varphi (\vec{e}_i)=\vec{e}_i$ and $\psi(\vec{e}_i)=0$ for every $i=1,\dots,n$, because $\sigma(1)=I, \sigma(0)=0$ and $\delta_j(1)=\delta_j(0)=0$ for any $j=1,\dots,n$. This shows that ${}^1{\bf{a}}={\bf{a}}$. Moreover, by Remark \ref{remark0} we obtain that
$${}^d({}^c {\bf{a}})={}^d( c^{-1}\textbf{a}\tilde{\sigma}(c)+c^{-1}\tilde{\delta}(c) )=d^{-1}( c^{-1}\textbf{a}\tilde{\sigma}(c)+c^{-1}\tilde{\delta}(c) )\tilde{\sigma}(d)+d^{-1}\tilde{\delta}(d)=$$
$$=d^{-1}c^{-1}\textbf{a}\tilde{\sigma}(c)\tilde{\sigma}(d)+d^{-1}c^{-1}\tilde{\delta}(c)\tilde{\sigma}(d)+d^{-1}\tilde{\delta}(d)=(cd)^{-1}\textbf{a}\tilde{\sigma}(cd)+(cd)^{-1}\tilde{\delta}(cd)\ , $$
i.e. ${}^d({}^c {\bf{a}})={}^{(cd)} {\bf{a}}$.
\end{proof}

\begin{definition}
Let ${\bf{a}}\in \mathbb{F}^{n}$. We define the \textit{right (left) $(\sigma,\delta)$-conjugacy class $[{\bf{a}}] \ \left([{\bf{a}}]_L\right)$ of ${\bf{a}}$} as the set
$$[{\bf{a}}]:=\{ {\bf{b}}\in\mathbb{F}^n\ | \ {\bf{b}}\sim {\bf{a}}\} \ \ \ \left( [{\bf{a}}]_L:=\{ {\bf{b}}\in\mathbb{F}^n\ | \ {\bf{b}}\sim_L {\bf{a}}\} \right)\ .$$
\end{definition}

\smallskip

Finally, using the notion of the right (left) $(\sigma,\delta)$-conjugation, we can give the next result that allows one to evaluate on the right (left) a product of two any skew polynomials.

\begin{theorem} \label{Productrule}
Consider two skew polynomials $F({\bf{x}}),G({\bf{x}}) \in \mathcal{R}$ and $\bf{a} \in \mathbb{F}^{n}$. 
\begin{enumerate}
\item[$1)$] \cite[Theorem 3]{evaluationandinterpolation} if $G({\bf{a}}) = 0$, then $(FG({\bf{x}}))({\bf{a}}) = 0$; if $G({\bf{a}}) \neq 0$ then $$(FG({\bf{x}}))({\bf{a}}) = F\left({\bf{a}}^{G({\bf{a}})}\right)\ G({\bf{a}})\ ;$$
\item[$2)$] if $G_L({\bf{a}}) = 0$, then ${\left((GF)({\bf{x}})\right)}_L({\bf{a}}) = 0$; if $G_L({\bf{a}}) \neq 0$ then $${\left((GF)({\bf{x}})\right)}_L({\bf{a}}) = G_L({\bf{a}})\ F_L\left( {}^{G_L({\bf{a}})}\bf{a}\right)\ .$$
\end{enumerate}
\end{theorem}

\begin{proof}
Suppose that $G_L({\bf{a}}) = 0$. Then $G(\textbf{x})=\sum_{i=1}^{n}(x_i-a_i)G_i(\textbf{x})$ and this gives $G(\textbf{x})F(\textbf{x})=\sum_{i=1}^{n}(x_i-a_i)\left(G_i(\textbf{x})F(\textbf{x})\right)$, that is, ${\left((GF)(\textbf{x})\right)}_L({\bf{a}}) =\left(G(\textbf{x})F(\textbf{x})\right)_L(\textbf{a})=0$. Assume now that $G_L({\bf{a}}) \neq 0$. Then we have
$$G(\textbf{x})F(\textbf{x}) = \left(\sum_{i=1}^{n}(x_i-a_i)G_i(\textbf{x})\right) F(\textbf{x}) + G_L(\textbf{a})F(\textbf{x})= $$
$$= \left(\sum_{i=1}^{n}(x_i-a_i)G_i(\textbf{x}) F(\textbf{x})\right) + G_L(\textbf{a}) \left(\sum_{i=1}^{n}(x_i-b_i)F_i(\textbf{x})\right)
+ G_L(\textbf{a})F_L(\bf{b})\ ,$$
where $\textbf{b}={}^{G_L({\bf{a}})}\bf{a}$ and $b_i:=\left( {}^{G_L({\bf{a}})}\bf{a}\right)_i$ for any $i=1,\dots,n$. By \eqref{lambda x}, we obtain that 
$$G_L({\bf{a}})\textbf{x}= \textbf{x}\tilde{\sigma}\left( G_L({\bf{a}}) \right)+\tilde{\delta}\left( G_L({\bf{a}}) \right)\ , \ \ G_L({\bf{a}}){}^{G_L({\bf{a}})}\textbf{a}=\textbf{a}\tilde{\sigma}\left( G_L({\bf{a}}) \right)+\tilde{\delta}\left( G_L({\bf{a}}) \right)\ .$$
Thus, we get $G_L({\bf{a}})(\textbf{x}-\textbf{a})=(\textbf{x}-\textbf{a})\tilde{\sigma}(G_L({\bf{a}}))$ and this gives 
$\left(G(\textbf{x})F(\textbf{x})\right)_L(\textbf{a})=G_L(\textbf{a})F_L({\bf{b}} )=G_L(\textbf{a})F_L( {}^{G_L({\bf{a}})}{\bf{a}} )$.
\end{proof}

\section{$(\sigma,\delta)$-Partial derivatives}\label{Sect2}

Based on properties of the ring $\mathcal{R}:=\mathbb{F}[\textbf{x};\sigma,\delta]$ and some results of the previous section, we can introduce the notion of right (left) $(\sigma,\delta)$-partial derivatives for multivariate skew polynomials that extend the right (left) univariate derivatives given in \cite[Definition 3.37]{resultantepaper}. Moreover, from now on, to deal with all the left division cases, we will always assume implicitly that $\varphi$, defined as in \eqref{6}, is an isomorphism of additive groups. 

\smallskip

Consider $F(\textbf{x})\in \mathcal{R}$ and fix, for instance, the graded lexicographic ordering $\prec$ ($\prec_L$) from right to left (from left to right) in the set $\mathcal{M}$ of monomials with $1\prec x_1\prec \dots \prec x_n$ ($1\prec_L x_1\prec_L \dots \prec_L x_n$). The existence of each $G_{i}(\textbf{x}) \ \left( G_{i,L}(\textbf{x})\right)$ as in \eqref{for 1.6} (\eqref{4}) is provided as usual  by a right (left) multivariate division algorithm. On the other hand, with respect to $\prec$ ($\prec_L$), by \cite[Lemma 5]{evaluationandinterpolation} (Lemma \ref{lema5bis}) such polynomials $G_{i}(\textbf{x}) \ \left( G_{i,L}(\textbf{x})\right)$ are unique for all $i=1, \dots ,n$.

\smallskip

This allows us to introduce the concept of right (left) $(\sigma,\delta)$-partial derivatives of any free multivariate skew polynomial which will be useful to prove the main results of the next section.

\begin{definition}\label{Der1.1}
Let $F(\textbf{x})\in\mathcal{R}$ and $\textbf{a}:=(a_1, \dots ,a_n)\in \mathbb{F}^{n}$. For all $i=1, \dots ,n$, we define the \textit{right (left) $(\sigma,\delta)$-partial derivative of $F({\bf{x}})$ at ${\bf{a}}$ with respect to the variable $x_i$}, denoted by $\Delta^{x_i}_{\textbf{a}}F(\textbf{a})\ \left(\ \left(\Delta^{x_i}_{\textbf{a},L}F\right)_L(\textbf{a})\ \right)$, as the right (left) evaluation at the point $\textbf{a}$ of the skew polynomial $\Delta^{x_i}_{\textbf{a}}F(\textbf{x})$ ($\Delta^{x_i}_{\textbf{a},L}F(\textbf{x})$) obtained by writing $$F=\sum_{i=1}^{n}\Delta^{x_i}_{\textbf{a}}F(\textbf{x})\cdot(x_i-a_i)+F(\textbf{a})\quad \left(F=\sum_{i=1}^{n}(x_i-a_i)\cdot\Delta^{x_i}_{\textbf{a},L}F(\textbf{x})+F_L(\textbf{a})\right)\ .$$
The skew polynomial $\Delta^{x_i}_{\textbf{a}}F(\textbf{x})\ \left( \Delta^{x_i}_{\textbf{a},L}F({\bf{x}})\right)$ in $\mathcal{R}$ will be called the \textit{right (left) $(\sigma,\delta)$-partial derivative polynomial of $F$ at $\bf{a}$ with respect to the variable $x_i$.} Moreover, we define $\Delta^{o}_{\textbf{a}}F(\textbf{x}):=F(\textbf{x})=:\Delta^{o}_{\textbf{a},L}F(\textbf{x})$ and $\Delta^{o}_{\textbf{a}}F(\textbf{a}):=F(\textbf{a})\ \left(\ \left(\Delta^{o}_{\textbf{a},L}F\right)_L(\textbf{a}):=F_L(\textbf{a})\ \right)$. 
\end{definition}

Furthermore, we can recursively define right (left)  $(\sigma,\delta)$-partial derivatives of higher order of a multivariate skew polynomial as follows.

\begin{definition}\label{Der1.2.}
Given $F(\textbf{x})\in \mathcal{R}$ and $\textbf{a} \in \mathbb{F}^{n}$, 
consider $m:=x_{i_1} x_{i_2} \cdots x_{i_s} \in \mathcal{M}$ with $i_l \in \{1, \dots ,n\}$ for all $l=1,\dots,s$.
We define recursively the \textit{right (left) $(\sigma,\delta)$-partial derivative of $F$ at $\bf{a}$ with respect to $m$}, denoted by $\Delta_{\textbf{a}}^{m}F(\textbf{a} )$ $\left(\ \left( \Delta_{\textbf{a},L}^{m}F\right)_L(\textbf{a})\ \right)$, as the right (left) evaluation at the point $\textbf{a}$ of the skew polynomial 
$$\Delta^{m}_{\textbf{a}}F(\textbf{x}):=\Delta^{x_{i_1}}_{\textbf{a}}\left( \Delta^{x_{i_2}}_{\textbf{a}} \left( \dots  (\Delta^{x_{i_s}}_{\textbf{a}}F ) \right) \right)(\textbf{x})$$
$$\left(\ \Delta^{m}_{\textbf{a},L}F(\textbf{x}):=\Delta^{x_{i_s}}_{\textbf{a},L}\left( \Delta^{x_{i_{s-1}}}_{\textbf{a},L} \left( \dots  (\Delta^{x_{i_1}}_{\textbf{a},L}F ) \right) \right)(\textbf{x})\ \right) \ ,$$
that is,
$$\Delta^{m}_{\textbf{a}}F(\textbf{a}):=\Delta^{x_{i_1}}_{\textbf{a}}\left( \Delta^{x_{i_2}}_{\textbf{a}} \left( \dots  (\Delta^{x_{i_s}}_{\textbf{a}}F ) \right) \right)(\textbf{a}) \ ,$$
$$ \left(\Delta^{m}_{\textbf{a},L}F\right)_L(\textbf{a}):=\left(\Delta^{x_{i_s}}_{\textbf{a},L}\left( \Delta^{x_{i_{s-1}}}_{\textbf{a},L} \left( \dots  (\Delta^{x_{i_1}}_{\textbf{a},L}F ) \right) \right)\right)_{L}(\textbf{a})\ .$$
\end{definition}

\begin{remark}

Let $\mathbb{F}$ be a division ring, $\textbf{a}:=(a_1,a_2, \dots ,a_n)\in \mathbb{F}^{n}$ and assume that $\sigma=Id$ and $\delta=0$. Given $F(\textbf{x})=x_1^2x_2\in \mathbb{F}[\textbf{x};{Id},0]$  and fixed the above graded lexicographic order $\prec$ ($\prec_L$) over $\mathbb{F}[\textbf{x};{Id},0]$, we have $$F(\textbf{x})=(a_2x_1+a_2a_1)(x_1-a_1)+x_1^2(x_2-a_2)+a_2a_1^2$$
$$\left(\ F(\textbf{x})=(x_1-a_1)(x_1x_2+a_1x_2)+(x_2-a_2)a_1^2+a_2a_1^2\ \right)\ .$$
Then, we obtain that
$$\Delta_{\textbf{a}}^{x_1}F(\textbf{x})=a_2x_1+a_2a_1 \;(\Delta_{\textbf{a},L}^{x_1}F(\textbf{x})=x_1x_2+a_1x_2)\;,\; \Delta_{\textbf{a}}^{x_2}F(\textbf{x})=x_1^2\; (\Delta_{\textbf{a},L}^{x_2}F(\textbf{x})=a_1^2).$$
Hence $\Delta_{\textbf{a}}^{x_1}F(\textbf{a})=(\Delta_{\textbf{a},L}^{x_1}F)_L(\textbf{a})=2a_2a_1$ and $\Delta_{\textbf{a}}^{x_2}F(\textbf{a})=(\Delta_{\textbf{a},L}^{x_2}F)_L(\textbf{a})=a_1^2$. Note that if $\mathbb{F}$ is a field, then we can write $\Delta_{\textbf{a}}^{x_1}F(\textbf{a})=(\Delta_{\textbf{a},L}^{x_1}F)_L(\textbf{a})=2a_1a_2$ and $\Delta_{\textbf{a}}^{x_2}F(\textbf{a})=(\Delta_{\textbf{a},L}^{x_2}F)_L(\textbf{a})=a_1^2$ , that is, the notion of right (left) $(Id,0)$-partial derivative is of classical type. On the other hand, we have  $\Delta_{\textbf{a}}^{x_1x_2}F(\textbf{a})=2a_1 \neq \Delta_{\textbf{a}}^{x_2x_1}F(\textbf{a})=0$ $\left(\ \left(\Delta_{\textbf{a},L}^{x_2x_1}F\right)_L(\textbf{a})=0 \neq \left(\Delta_{\textbf{a},L}^{x_1x_2}F\right)_L(\textbf{a})=2a_1\ \right)$. This shows that, unlike in the classical case, in general the mixed right (left) $(\sigma,\delta)$-partial derivatives of a multivariate skew polynomial are not equal.
\end{remark}

\begin{remark}
Every skew polynomial $F(\textbf{x})\in \mathcal{R}$ can be written in terms of its right and left $(\sigma,\delta)$-partial derivatives. More precisely, the definitions of right (left) partial derivatives and a recursive use of right (left) divisions allow us to obtain a right (left) skew multivariate Taylor-type expansion of every $F(\textbf{x})\in \mathcal{R}$ centered at any point ${\bf{a}}=(a_1,...,a_n)\in \mathbb{F}^{n}$ as follows. For any variables $x_i,x_j$ with $i,j\in \{1,2,...,n\}$, we have  
$$F({\bf{x}})=\sum_{i_1=1}^{n}\Delta^{x_{i_1}}_{\textbf{a}}F({\bf{x}})\ (x_{i_1}-a_{i_1})+F(\textbf{a})\ ,$$ 
$$\Delta^{x_{i_1}}_{\textbf{a}}F(x)=\sum_{i_2=1}^{n}\Delta^{x_{i_2}x_{i_1}}_{\textbf{a}}F({\bf x})\ (x_{i_2}-a_{i_2})+\Delta^{x_{i_1}}_{\textbf{a}}F(\textbf{a})\ .$$ Then, substituting $\Delta^{x_i}_{\textbf{a}}F(x)$ in $F(\textbf{x})$, it follows that 
$$F(\textbf{x})=\sum_{i_1=1}^{n}\left(\sum_{i_2=1}^{n}\Delta_{{\bf{a}}}^{x_{i_2}x_{i_1}}F({\bf{x}})(x_{i_2}-a_{i_2})(x_{i_1}-a_{i_1})\right)+\sum_{i_1=1}^{n}\Delta_{{\bf{a}}}^{x_{i_1}}F({\bf{a}})(x_{i_1}-a_{i_1})+F({\bf{a}})\ .$$
So, by a recursive argument, one can deduce the following identities:\begin{equation}\label{for2.5}
F({\bf{x}})=\sum_{k=1}^{\deg F}\left(\sum_{i_1,i_2,...,i_k=1}^{n}\Delta_{{\bf{a}}}^{x_{i_k}\cdots\ x_{i_2}x_{i_1}}F({\bf{a}})\ (x_{i_k}-a_{i_k})\cdots (x_{i_2}-a_{i_2})(x_{i_1}-a_{i_1}) \right)+F({\bf{a}})\ ,
\end{equation}
{\small
\begin{equation}\label{for2.5bis}
    F({\bf{x}})=\sum_{k=1}^{\deg F}\left(\sum_{i_1,i_2,...,i_k=1}^{n}(x_{i_1}-a_{i_1})(x_{i_2}-a_{i_2})\cdots (x_{i_k}-a_{i_k})\ \left(\Delta_{{\bf{a}},L}^{x_{i_1}x_{i_2}\cdots\ x_{i_k}}F\right)_L({\bf{a}}) \right)+F_L({\bf{a}})\ .
\end{equation}
}
For instance, some $F({\bf{x}}), G({\bf x})\in\mathcal{R}$ such that $$F({\bf{a}})=\alpha,\ \Delta_{{\bf{a}}}^{x_{1}}F({\bf{a}})=\beta,\ \Delta_{{\bf{a}}}^{x_{2}}F({\bf{a}})=\gamma,\ \Delta_{{\bf{a}}}^{x_{1}x_{1}}F({\bf{a}})=\delta,\ \Delta_{{\bf{a}}}^{x_{1}x_{2}}F({\bf{a}})=\epsilon\ ,$$ 
{\small
$${\tiny{G_L({\bf{a}})=\alpha,\ \left(\Delta_{{\bf{a}},L}^{x_{1}}G\right)_L({\bf{a}})=\beta,\ \left(\Delta_{{\bf{a}},L}^{x_{2}}G\right)_L({\bf{a}})=\gamma,\ \left(\Delta_{{\bf{a}},L}^{x_{1}x_{1}}G\right)_L({\bf{a}})=\delta,\ \left(\Delta_{{\bf{a}},L}^{x_{1}x_{2}}G\right)_L({\bf{a}})=\epsilon\ }},$$
}
with $\alpha,\beta,\gamma,\delta,\epsilon\in\mathbb{F}$, are given by
$$F({\bf{x}})=\epsilon (x_{1}-a_{1})(x_{2}-a_{2})+\delta (x_{1}-a_{1})^2+\gamma (x_{2}-a_{2})+\beta (x_{1}-a_{1})+\alpha\ ,$$
$$G({\bf{x}})=(x_{1}-a_{1})(x_{2}-a_{2})\epsilon+(x_{1}-a_{1})^2\delta+(x_{2}-a_{2})\gamma+ (x_{1}-a_{1})\beta+\alpha\ .$$
\end{remark}

\smallskip

The linearity of the right (left) $(\sigma,\delta)$-partial derivatives is shown in the next result.  
  
\begin{lemma}\label{Der1.5}
Let $F({\bf{x}}),G({\bf{x}}) \in \mathcal{R}$, ${\bf{a}}\in \mathbb{F}^{n}$, $\lambda \in \mathbb{F}$ and $m \in \mathcal{M}$. Then the following properties hold: 
\begin{itemize}
    \item[$1)$] $\Delta^{m}_{\bf{a}}\lambda=0$, \ $\Delta^{m}_{\bf{a}}(\lambda F({\bf{x}})+G({\bf{x}}))=\lambda(\Delta^{m}_{\bf{a}}F({\bf{x}}))+\Delta^{m}_{\bf{a}}G({\bf{x}})$;
    \item[$2)$] $\Delta^{m}_{{\bf{a}},L}\lambda=0$, \ $\Delta^{m}_{{\bf{a}},L}( F({\bf{x}})\lambda+G({\bf{x}}))=(\Delta^{m}_{{\bf{a}},L}F({\bf{x}}))\lambda+\Delta^{m}_{{\bf{a}},L}G({\bf{x}})$.
    \end{itemize}
\end{lemma}

\begin{proof}
We only prove $1)$, because the proofs of $2)$ are analogous, provided that $\varphi$ is an additive group isomorphism. Since $\lambda=\sum_{k=1}^{n}0\cdot(x_i-a_i)+\lambda$ it follows that $\Delta^{x_i}_{\textbf{a}}\lambda=0$ for all $i=1, \dots ,n$. Therefore, for any $m \in \mathcal{M}$, we have $\Delta^{m}_{\textbf{a}}\lambda=0$. 
Finally, for the second identity in $1)$, it is sufficient to show that it is true for any variable $x_i$ with $i=1,\dots,n$. Note that we can write
$$ \lambda F(\textbf{x})=\sum_{i=1}^{n}\lambda \Delta^{x_i}_{\textbf{a}}F(\textbf{x})\cdot(x_i-a_i)+\lambda F(\textbf{a}), \ G(\textbf{x})=\sum_{i=1}^{n}\Delta^{x_i}_{\textbf{a}}G(\textbf{x})\cdot(x_i-a_i)+G(\textbf{a})\ .
$$
Then, we get
$$\lambda F(\textbf{x})+G(\textbf{x})=\sum_{i=1}^{n}(\lambda \Delta^{x_i}_{\textbf{a}}F(\textbf{x})+\Delta^{x_i}_{\textbf{a}}G(\textbf{x}))\cdot(x_i-a_i)+\lambda F(\textbf{a})+G(\textbf{a})\ .$$ Thus, for each variable $x_i$, we have $\Delta^{x_i}_{\textbf{a}}(\lambda F(\textbf{x})+G(\textbf{x}))=\lambda \Delta^{x_i}_{\textbf{a}}F(\textbf{x})+\Delta^{x_i}_{\textbf{a}}G(\textbf{x})$ and, by a recursive argument, we obtain the second identity of $1)$ for any $m\in \mathcal{M}$.
\end{proof}

By using Lemma \ref{Der1.5}, we can obtain a useful formula to compute the right (left) $(\sigma,\delta)$-partial derivatives of a product of two multivariate skew polynomials in $\mathcal{R}$.

\begin{lemma}\label{Der1.7}
Given $F({\bf{x}}),G({\bf{x}}) \in \mathcal{R}$, ${\bf{a}}\in \mathbb{F}^{n}$, $\lambda \in \mathbb{F}$ and $m:=x_{i_1}\cdots x_{i_s} \in \mathcal{M}$ with $i_l\in \{1, \dots ,n\}$ for all $l=1, \dots ,n$, we have
\begin{eqnarray}\label{For2.1}
\Delta^{m}_{{\bf{a}}}(F\cdot G)({\bf{x}})= F({\bf{x}}) \cdot  \Delta^{m}_{{\bf{a}}}G({\bf{x}}) + \sum_{k=1}^{s}\Delta_{{\bf{a}}}^{x_{i_1}\cdots x_{i_k}}\left(F \cdot {\Delta_{{\bf{a}}}^{x_{i_{k+1}}\cdots x_{i_s}}G({\bf{a}})}\right)({\bf{x}})\;,\end{eqnarray}
\begin{eqnarray}\label{For2.2}
\qquad
\Delta^{m}_{{\bf{a}},L}(F \cdot G)({\bf{x}}) =   \Delta^{m}_{{\bf{a}},L}F({\bf{x}}) \cdot G({\bf{x}}) + \sum_{k=1}^{s}\Delta_{{\bf{a}},L}^{x_{i_1}\cdots x_{i_k}}\left(
{\left(\Delta_{{\bf{a}},L}^{x_{i_{k+1}}\cdots x_{i_s}}F\right)}_L({\bf{a}})\cdot G\right)({\bf{x}})\;.\end{eqnarray}
\end{lemma}

\begin{proof}
Let us show that \eqref{For2.1} holds for any variable $x_i$ with $i=1, \dots ,n$. Indeed, we have 
{\small
 \begin{align*}
F(\textbf{x}) \cdot G(\textbf{x}) &= \sum_{i=1}^{n}F(\textbf{x})\cdot \Delta_{\textbf{a}}^{x_i}G(\textbf{x})\cdot (x_i-a_i)+F(\textbf{x})\cdot G(\textbf{a})\\
&= \sum_{i=1}^{n}F(\textbf{x})\cdot \Delta_{\textbf{a}}^{x_i}G(\textbf{x})\cdot (x_i-a_i)+\sum_{i=1}^{n}\Delta_{\textbf{a}}^{x_i}(F\cdot G(\textbf{a}))(\textbf{x})(x_i-a_i)+(F\cdot G(\textbf{a}))(\textbf{a})\\
&= \sum_{i=1}^{n}\left[F(\textbf{x})\cdot \Delta_{\textbf{a}}^{x_i}G(\textbf{x})+\Delta_{\textbf{a}}^{x_i}(F\cdot G(\textbf{a}))(\textbf{x})\right](x_i-a_i)+(F\cdot G(\textbf{a}))(\textbf{a})
\end{align*}
}
Thus, $
    \Delta_{\textbf{a}}^{x_i}(F \cdot G)(\textbf{x})=F(\textbf{x})\cdot \Delta_{\textbf{a}}^{x_i}G(\textbf{x})+\Delta_{\textbf{a}}^{x_i}(F\cdot G(\textbf{a}))(\textbf{x})$. Suppose \eqref{For2.1} holds for any monomial $m'\in \mathcal{M}$ such that $\deg(m')\leq s-1$ and $s\geq 2$. Then, we have 
{\small    
 \begin{align*}
 \Delta_{\textbf{a}}^{m}(F\cdot G)(\textbf{x}) &=\Delta_{\textbf{a}}^{x_{i_1}}\left(\Delta_{\textbf{a}}^{x_{i_2}\cdots x_{i_s}}(F\cdot G)\right)(\textbf{x})\\
 &=\Delta_{\textbf{a}}^{x_{i_1}}\left(F \cdot  \Delta^{x_{i_2}\cdots x_{i_s}}_{\textbf{a}}G + \sum_{k=2}^{s}\Delta_{\textbf{a}}^{x_{i_2}\cdots x_{i_k}}\left(F \cdot {\Delta_{\textbf{a}}^{x_{i_{k+1}}\cdots x_{i_s}}G(\textbf{a})}\right)\right)(\textbf{x})\\
  &=\Delta_{\textbf{a}}^{x_{i_1}}\left(F \cdot  \Delta^{x_{i_2}\cdots x_{i_s}}_{\textbf{a}}G\right)(\textbf{x}) + \Delta_{\textbf{a}}^{x_{i_1}}\left( \sum_{k=2}^{s}\Delta_{\textbf{a}}^{x_{i_2}\cdots x_{i_k}}\left(F \cdot \Delta_{\textbf{a}}^{x_{i_{k+1}}\cdots x_{i_s}}G(\textbf{a})\right)\right)(\textbf{x})\\
   &=  F(\textbf{x}) \cdot  \Delta^{m}_{\textbf{a}}G(\textbf{x}) + \sum_{k=1}^{s}\Delta_{\textbf{a}}^{x_{i_1}\cdots x_{i_k}}\left(F \cdot {\Delta_{\textbf{a}}^{x_{i_{k+1}} \cdots x_{i_s}}G(\textbf{a})}\right)(\textbf{x})\;, 
\end{align*}
}
where the third equality is due to the linearity of the right $(\sigma,\delta)$-partial derivatives and the last equality is due to the inductive hypothesis. Finally, by similar arguments used for \eqref{For2.1}, one can prove that formula \eqref{For2.2} holds.
\end{proof}

\begin{remark}
By Lemma \ref{Der1.7}, if $\mathbb{F}$ is a field, $\sigma={Id}$ and $\delta=0$, then we get
$$\Delta_{\textbf{a}}^{x_i}(F \cdot G)(\textbf{a})=F(\textbf{a})\cdot \Delta_{\textbf{a}}^{x_i}G(\textbf{a})+\Delta_{\textbf{a}}^{x_i}F(\textbf{a})\cdot G(\textbf{a})\ ,
$$
$$\left(\Delta_{\textbf{a},L}^{x_i}(F \cdot G)\right)_L(\textbf{a})=F_L(\textbf{a})\cdot \left(\Delta_{\textbf{a},L}^{x_i}G\right)_L(\textbf{a})+\left(\Delta_{\textbf{a},L}^{x_i}F\right)_L(\textbf{a})\cdot G_L(\textbf{a})\ ,
$$
which correspond to the classical Leibniz's ruler.
\end{remark}

\section{Skew Hermite-type interpolation problem}\label{Sect3}

In this section, using the notion of the right (left) $(\sigma,\delta)$-partial derivatives introduced in Section \ref{Sect2} (see Definition \ref{Der1.1}), we solve a skew Hermite-type interpolation problem in $\mathcal{R}$ which extends the skew Lagrange interpolation problem of \cite[Theorem 4]{evaluationandinterpolation}. 

\subsection{The multivariate case $(n\geq 1)$}
Let us start by defining a left (right) ideal associated to a finite set of points in $\mathbb{F}^{n}$ which will be crucial for the skew Hermite-type interpolation problem.

\begin{definition}\label{def-ideals}
Let $\Omega=\{{\bf{a_1}},{\bf{a_2}},\dots,{\bf{a_k}}\} \subseteq \mathbb{F}^n$ be a finite set and consider the monomials $m_j=x_{j_{s(j)}}\cdots x_{j_2}x_{j_1} \in \mathcal{M}$ such that either $j_l \in \{1,2, \dots ,n\}$ and $s(j)\geq 1$, or $m_j=o$. Given $\vec{m}:=(m_1,m_2, \dots ,m_k)\in\mathcal{M}^k$, we denote by $I^{\vec{m}}(\Omega)$ the set
$$\{F \in \mathcal{R}: F({\bf{a_j}}) = \Delta_{{\bf{a_j}}}^{x_{j_1}}F({\bf{a_j}})=\Delta_{{\bf{a_j}}}^{x_{j_2}x_{j_1}}F({\bf{a_j}})= \dots = \Delta_{{\bf{a_j}}}^{m_j}F({\bf{a_j}})=0, \;\forall j=1,\dots,k \}\ ,$$ 
and by $I^{\vec{m}}_L(\Omega)$ the set
$$\{F \in \mathcal{R}: F_L({\bf{a_j}}) = \left(\Delta_{{\bf{a_j}},L}^{x_{j_{s(j)}}}F\right)_L({\bf{a_j}})= \dots = \left(\Delta_{{\bf{a_j}},L}^{x_{j_{s(j)}} \cdots\ x_{j_2}}F\right)_L({\bf{a_j}})= \left(\Delta_{{\bf{a_j}},L}^{m_j}F\right)_L({\bf{a_j}})=0, $$
$\forall j=1,\dots,k \}\ .$
\end{definition}

\begin{proposition}
For any finite set $\Omega=\{{\bf{a_1}}, \dots , {\bf{a_k}}\} \subseteq \mathbb{F}^{n}$ and $\vec{m}=(m_1, \dots , m_k)\in\mathcal{M}^k$, the set $I^{\vec{m}}(\Omega)$ $\left(I^{\vec{m}}_L(\Omega)\right)$ is a left (right) ideal in $\mathcal{R}$.
\end{proposition}
\begin{proof}
Given $F,G\in I^{\vec{m}}(\Omega)$ $\left(I^{\vec{m}}_L(\Omega)\right)$, by Lemma \ref{Der1.5} we have $F+G\in I^{\vec{m}}(\Omega)$ $\left(I^{\vec{m}}_L(\Omega)\right)$. Finally, for any $F\in \mathcal{R}$ and $G\in I^{\vec{m}}(\Omega)$ $\left(I^{\vec{m}}_L(\Omega)\right)$, by Lemma \ref{Der1.7} and Theorem \ref{Productrule} we obtain that $F\cdot G\in I^{\vec{m}}(\Omega)$ $\left(G\cdot F\in I^{\vec{m}}_L(\Omega)\right)$. Therefore, $I^{\vec{m}}(\Omega)$ $\left(I^{\vec{m}}_L(\Omega)\right)$ is a left (right) ideal in $\mathcal{R}$.
\end{proof}

\begin{remark}\label{rem0}
When all the components of $\vec{m}$ are equal, for instance $\vec{m}=(m',\dots,m',\dots)$ with $m':=x_{j_s}\dots x_{j_2}x_{j_1}\in\mathcal{M}$, one could define
$I^{\vec{m}}(\Omega)=I^{(m')}(\Omega) \ \left(\ I_L^{\vec{m}}(\Omega)=I_L^{(m')}(\Omega)\ \right)$ as the set
{\small
$$\{F \in \mathcal{R}: F({\bf{a}}) = \Delta_{{\bf{a}}}^{x_{j_1}}F({\bf{a}})=\Delta_{{\bf{a}}}^{x_{j_2}x_{j_1}}F({\bf{a}})= \dots = \Delta_{{\bf{a}}}^{m_j}F({\bf{a}})=0, \;\forall {\bf a}\in\Omega \}$$
$$\left(\{F \in \mathcal{R}: F_L({\bf{a}}) = \left(\Delta_{{\bf{a}}}^{x_{j_s}}F\right)_L({\bf{a}})=\dots=\left(\Delta_{{\bf{a}}}^{x_{j_s}\dots x_{j_2}}F\right)_L({\bf{a}})= \left(\Delta_{{\bf{a}}}^{m_j}F\right)_L({\bf{a}})=0, \;\forall {\bf a}\in\Omega \}\right)\ ,$$
}
for any $\Omega\subseteq\mathbb{F}^n$. In particular, when
$\vec{m}=(o,\dots,o,\dots)$, recalling that 
$$\Delta_{{\bf{a}}}^{o}F({\bf{a}})=F({\bf{a}})\ \ \ \left(\ \left(\Delta_{{\bf{a}},L}^{o}F\right)_L({\bf{a}})=F_L({\bf{a}})\ \right)\ ,$$ 
we obtain the left ideal (as in \cite[Definition 13]{evaluationandinterpolation}) 
$$I^{(o)}(\Omega)=I(\Omega):=\{F\in \mathcal{R}:F({\bf{a}})=0,\;\forall {\bf{a}}\in \Omega\}\ ,$$
and the right ideal 
$$I^{(o)}_L(\Omega)=I_L(\Omega):=\{F\in \mathcal{R}:F_L({\bf{a}})=0,\;\forall {\bf{a}}\in \Omega\}\ .$$
\end{remark}

\medskip

\begin{definition}\label{def-zeros}
Let $I\subset\mathcal{R}$ be a left (right) ideal
and consider $\vec{m}:=(m_1,m_2, \dots ,m_k)$ with
monomials $m_j=x_{j_{s(j)}}\cdots x_{j_2}x_{j_1} \in \mathcal{M}$ for every $j=1,\dots,k$, such that either $j_l \in \{1,2, \dots ,n\}$ and $s(j)\geq 1$, or $m_j=o$. We denote by $Z^{\vec{m}}(I)$ the set
$$\{{\bf a} \in \mathbb{F}^n\ :\  F({\bf{a}}) = \Delta_{{\bf{a}}}^{x_{j_1}}F({\bf{a}})=\Delta_{{\bf{a}}}^{x_{j_2}x_{j_1}}F({\bf{a}})= \dots = \Delta_{{\bf{a}}}^{m_j}F({\bf{a}})=0, $$
\qquad\qquad\quad $\forall j=1,\dots,k,\ \forall F\in I \}\ ,$

\medskip

\noindent and by $Z^{\vec{m}}_L(I)$ the set
$$\{{\bf a} \in \mathbb{F}^n\ :\ F_L({\bf{a}}) = \left(\Delta_{{\bf{a}},L}^{x_{j_{s(j)}}}F\right)_L({\bf{a}})= \dots = \left(\Delta_{{\bf{a}},L}^{x_{j_{s(j)}} \cdots\ x_{j_2}}F\right)_L({\bf{a}})= \left(\Delta_{{\bf{a}},L}^{m_j}F\right)_L({\bf{a}})=0, $$
\ \ \ $\forall j=1,\dots,k,\ \forall F\in I \}\ .$
\end{definition}

\medskip

\begin{remark}
The following properties hold:
\begin{enumerate}
\item if $I_1\subseteq I_2\subseteq \mathcal{R}$, then for any $\vec{m}\in\mathcal{M}^k$ we have
$$Z^{\vec{m}}(I_2)\subseteq Z^{\vec{m}}(I_1) \ \ \left( Z^{\vec{m}}_L(I_2)\subseteq Z^{\vec{m}}_L(I_1)\right) \ ;$$
\item consider $\vec{m}:=(m_1, \dots ,m_k)$ with
monomials $m_j=x_{j_{s(j)}}\cdots x_{j_2}x_{j_1} \in \mathcal{M}$ for every $j=1,\dots,k$, such that either $j_l \in \{1,2, \dots ,n\}$ and $s(j)\geq 1$, or $m_j=o$; if $\vec{m}':=(m'_1, \dots ,m'_k)$ with $m'_j=x_{j_{t(j)}}\cdots x_{j_2}x_{j_1} \in \mathcal{M}\ \left( m'_j=x_{j_{s(j)}}\cdots x_{j_{t(j)}}\in \mathcal{M}\right)$ for every $j=1,\dots,k$, such that either $j_l \in \{1,2, \dots ,n\}$ and $1\leq t(j)\leq s(j)$, or $m'_j=o$, then for any left (right) ideal $I\subset\mathcal{R}$, we get
$$Z^{\vec{m}}(I)\subseteq Z^{\vec{m}'}(I) \ \ \left( Z^{\vec{m}}_L(I)\subseteq Z^{\vec{m}'}_L(I)\right) \ .$$
\end{enumerate}
\end{remark}

\smallskip

For simplicity, from now on we will only give proofs for the right versions of all the next results. The proofs of their relative left versions are very similar to that given below, keeping in mind the suitable changes of notation. 

\smallskip

The following technical result will be the key tool of all the next results for solving the skew Hermite-type interpolation problem (Theorem \ref{theorem 3.1.7}). 

\begin{lemma}\label{3.1.8}
Let $\Omega=\{{\bf{a_1}},\dots,{\bf{a_k}}\} \subseteq \mathbb{F}^n$ be a finite set and let $\vec{m}=(m_1,\dots,m_k)\in \mathcal{M}^k$ with either $m_j=o$, or $m_j=x_{j_{s(j)}}\cdots x_{j_2}x_{j_1}$ for all $j=1, \dots , k$ and $j_l \in \{1,\dots,n\}$. Then for any $m\in\mathcal{M}$ with $\deg (m)=N:=\sum_{i=1}^{k}\left[\deg (m_i)+1\right]$, there exists $F\in I^{\vec{m}}(\Omega)\ \left(I^{\vec{m}}_L(\Omega)\right)$ such that $\deg (F) =N$ and $LM(F)=m$. 

\noindent In particular, there always exists $F\in I^{\vec{m}}(\Omega)\ \left(I^{\vec{m}}_L(\Omega)\right)$ such that $0<\deg (F)\leq N$.
\end{lemma}

\begin{proof}
Write $m=x_{k_N}\ {\cdots}\ x_{k_2} x_{k_1}\in\mathcal{M}$. Defining $F_1({\bf{x}}):=x_{k_1}-({\bf a_1})_{k_1}$, we have $F_1({\bf a_1})=0$ with $\deg (F_1)=1$ and $\Delta_{{\bf a_1}}^{x_{1_1}}F_1({\bf{x}})=0$ or $1$.
Then, define $F_2({\bf{x}}):=\left( x_{k_2}-({\bf a_1})_{k_2}\right)F_1({\bf{x}})$. By Lemma \ref{Der1.7}, we have
$$\Delta_{{\bf a_1}}^{x_{1_1}}F_2({\bf{x}})=\left(x_{k_2}-({\bf a_1})_{k_2}\right)\Delta_{{\bf a_1}}^{x_{1_1}}F_1({\bf{x}})\ .$$
In any case, we get $F_2({\bf a_1})=\Delta_{{\bf a_1}}^{x_{1_1}}F_2({\bf a_1})=0$ with $\deg (F_2)=2$ and $\Delta_{{\bf a_1}}^{x_{1_2} x_{1_1}}F_2({\bf{x}})=0$ or $1$. By a recursive argument, we can construct $F_t({\bf{x}}):=\Pi_{i=1}^{t}\left(x_{k_i}-({\bf a_1})_{k_i}\right)$ with $t=\deg (m_1)+1$ and such that $F_t({\bf{x}})\in I^{m_1}(\{{\bf a_1}\})$. Thus, define now $G_1({\bf{x}}):=(x_{k_{t+1}}-\alpha_1)F_t({\bf{x}})$ for some $\alpha_1\in\mathbb{F}$. If $F_t({\bf a_2})=0$, then $G_1({\bf{x}})\in I^{m_1,o}(\{{\bf a_1}\}\cup \{{\bf a_2}\})$. Otherwise, by taking $\alpha_1:=\left({\bf a_2}^{F_t({\bf a_2})} \right)_{k_{t+1}}$, we get again $G_1({\bf{x}})\in I^{m_1,o}(\{{\bf a_1}\}\cup \{{\bf a_2}\})$ with $\deg (G_1)=t+1$. Therefore, defining $G_2({\bf{x}})=\left(x_{k_{t+2}}-\alpha_2 \right)G_1({\bf{x}})$, we have $G_2({\bf{x}})\in I^{m_1,o}(\{{\bf a_1}\}\cup \{{\bf a_2}\})$ and from Lemma \ref{Der1.7} it follows that
$$
\Delta_{{\bf a_2}}^{x_{2_1}}G_2({\bf{x}})=\left(x_{k_{t+2}}-\alpha_2 \right)\cdot \Delta_{{\bf a_2}}^{x_{2_1}}G_1({\bf{x}})\ .
$$
If $\Delta_{{\bf a_2}}^{x_{2_1}}G_1({\bf a_2})=0$, then $\Delta_{{\bf a_2}}^{x_{2_1}}G_2({\bf a_2})=0$. If $\Delta_{{\bf a_2}}^{x_{2_1}}G_1({\bf a_2})\neq 0$, then by choosing $\alpha_2:=\left({\bf a_2}^{\Delta_{{\bf a_2}}^{x_{2_1}}G_1({\bf a_2})}\right)_{k_{t+2}}$, we obtain $\Delta_{{\bf a_2}}^{x_{2_1}}G_2({\bf a_2})=0$. Hence, there exists
$G_2({\bf{x}})\in I^{m_1,x_{2_1}}(\{{\bf a_1}\}\cup \{{\bf a_2}\})$ with $\deg (G_2)=t+2$. By a recursive argument, we can find a skew polynomial $F({\bf{x}})\in I^{\vec{m}}(\Omega)$ such that $\deg (F)=N$ and $LM(F({\bf{x}}))=m$ by construction. 
\end{proof}

From the proof of Lemma \ref{3.1.8}, we can deduce the following algorithm for both versions.

\newpage

\begin{algorithm}[h!]
\small
\caption{
For any finite set $\Omega=\{{\bf{a_1}},\dots,{\bf{a_k}}\}\subseteq \mathbb{F}^{n}$ and every $m\in\mathcal{M}$ such that $\deg(m)=\sum_{i=1}^{k}(\deg(m_i)+1)$, compute 
$F\in I^{(m_1,\dots,m_k)}(\Omega)\ \left(\mathrm{or}\ F\in I_L^{(m_1,\dots,m_k)}(\Omega)\right)$ with $LM(F)=m$}
\label{Alg1}
\begin{algorithmic}[1]
\smallskip
\Require{
\smallskip
\phantom{aaaa}
\smallskip
\begin{itemize}
\item $\Omega=\{{\bf{a_1}},\dots,{\bf{a_k}}\}\subseteq \mathbb{F}^{n}$ 
\item $\vec{m}:=(m_1,\dots,m_k) \in \mathcal{M}^k$ with $m_j:=x_{j_{s(j)}}\dots x_{j_2}x_{j_1}$ for every $j\in\{1,\dots,k\}$, 
\item $m\in\mathcal{M}$ written as 
$$m=x_{(k,s(k))}\dots x_{(k,1)}x_{(k,0)}\dots x_{(2,s(2))}\dots x_{(2,1)}x_{(2,0)}\dots x_{(1,s(1))}\dots x_{(1,1)}x_{(1,0)}$$
\end{itemize}
\smallskip
}
\Ensure{$F\in I^{\vec{m}}(\Omega)\ \left(\ or \ \  F\in I^{\vec{m}}_L(\Omega)\right)$ with $\deg (F)=\sum_{i=1}^{k}\left[\deg (m_i)+1\right]$ and $LM(F)=m$.
\smallskip}
\State{$G_0\gets 1$}
\For{$j\gets 1$ to $k$}
\If{$G_{j-1}({\bf{a_j}})=0$ \ \ $\left(\ or \ \  \left(G_{j-1}\right)_L({\bf{a_j}})=0\ \right)$}
\State{$G_j\gets x_{(j,0)}\cdot G_{j-1}$ \ \ $\left(\ or \ \ G_j\gets G_{j-1}\cdot x_{(j,0)}\ \right)$}
\Else
\State{$G_j\gets \left[x_{(j,0)}-\left( {\bf{a_j}}^{G_{j-1}({\bf{a_j}})}\right)_{(j,0)}\right]\cdot G_{j-1}$ \\ \quad \quad \ \ \ $\left(\ or \ \  G_j\gets G_{j-1}\cdot \left[x_{(j,0)}-\left( {}^{\left(G_{j-1}\right)_L({\bf{a_j}})}{\bf{a_j}}\right)_{(j,0)}\right]\ \right)$}
\EndIf
\For{$i\gets 1$ to $s(j)$}
\If{$\Delta_{\bf{a_j}}^{x_{j_i}\dots\ x_{j_1}}G_j({\bf{a_j}})=0$ \ \ $\left(\ or \ \  \left(\Delta_{\bf{a_j},L}^{x_{j_{s(j)}}\dots\ x_{j_{s(j)+1-i}}}G_j\right)_L({\bf{a_j}})=0\ \right)$} 
\State{$G_j\gets x_{(j,i)}\cdot G_{j}$ \ \ $\left(\ or \ \  G_j\gets G_{j}\cdot x_{(j,i)} \ \right)$}
\Else
\State{$G_j\gets \left[x_{(j,i)}-\left( {\bf{a_j}}^{\Delta_{{\bf{a_j}}}^{x_{j_i}\dots\ x_{j_1}}G_{j}({\bf{a_j}})}\right)_{(j,i)}\right]\cdot G_{j}$ \\ \phantom{aaaaaaaa} $\left(\ or \ \  G_j\gets G_j\cdot \left[x_{(j,i)}-\left( {}^{\left(\Delta_{\bf{a_j},L}^{x_{j_{s(j)}}\dots\ x_{j_{s(j)+1-i}}}G_j\right)_L({\bf{a_j}})}{\bf{a_j}}\right)_{(j,i)}\right]\ \right)$}
\EndIf
\EndFor
\EndFor \\
\Return{$G_k$}
\medskip
\end{algorithmic}
\end{algorithm}

\medskip

Furthermore, Algorithm \ref{Alg1} can be modified as follows.

\begin{algorithm}[h!]
\small
\caption{Compute $F\in I^{(m_1,\dots,m_k)}(\Omega)\ \left(\mathrm{or}\ F\in I_L^{(m_1,\dots,m_k)}(\Omega)\right)$ such that $\deg(F)\leq \sum_{i=1}^{k}(\deg(m_i)+1)$}
\label{Alg2}
\begin{algorithmic}[1]
\Require{$\Omega=\{{\bf{a_1}},\dots,{\bf{a_k}}\}\subseteq \mathbb{F}^{n},m_1,\dots,m_k \in \mathcal{M}.$}
\algstore{myalg}
\end{algorithmic}
\end{algorithm}

\pagebreak

\begin{algorithm}[h!]                     
\begin{algorithmic} [1]                  
\algrestore{myalg}
\Ensure{ $F\in I^{(m_1,\dots,m_k)}(\Omega)$ with $\deg(F)\leq \sum_{i=1}^{k}(\deg(m_i)+1)$
\bigskip}
\State{$G_0\gets 1$}
\For{$j\gets 1$ to $k$}
\If{$G_{j-1}({\bf{a_j}})=0$ \ \ $\left(\ or \ \  \left(G_{j-1}\right)_L({\bf{a_j}})=0\ \right)$}
\State{$G_j\gets G_{j-1}$}
\Else
\State{$G_j\gets \left[x_{t}-\left( {\bf{a_j}}^{G_{j-1}({\bf{a_j}})}\right)_{t}\right]\cdot G_{j-1}$ \ \ $\left(\ or \ \  G_j\gets G_{j-1}\cdot \left[x_{t}-\left( {}^{\left(G_{j-1}\right)_L({\bf{a_j}})}{\bf{a_j}}\right)_{t}\right]\ \right)$}
\EndIf
\For{$i\gets 1$ to $s(j)$}
\If{$\Delta_{\bf{a_j}}^{x_{j_i}\dots\ x_{j_1}}G_j({\bf{a_j}})=0$ \ \ $\left(\ or \ \  \left(\Delta_{\bf{a_j},L}^{x_{j_{s(j)}}\dots\ x_{j_{s(j)+1-i}}}G_j\right)_L({\bf{a_j}})=0\ \right)$} 
\State{$G_j\gets G_{j}$ }
\Else
\State{$G_j\gets \left[x_{t}-\left( {\bf{a_j}}^{\Delta_{{\bf{a_j}}}^{x_{t}\dots\ x_{j_1}}G_{j}({\bf{a_j}})}\right)_{t}\right]\cdot G_{j}$ \\ \phantom{aaaaaaaa} $\left(\ or \ \  G_j\gets G_j\cdot \left[x_{t}-\left( {}^{\left(\Delta_{\bf{a_j},L}^{x_{j_{s(j)}}\dots\ x_{j_{s(j)+1-i}}}G_j\right)_L({\bf{a_j}})}{\bf{a_j}}\right)_{t}\right]\ \right)$}
\EndIf
\EndFor
\EndFor \\
\Return{$G_k$}
\end{algorithmic}
\end{algorithm}

\begin{lemma}\label{3.1.9}
Let $\Omega=\{{\bf{a_1}},\dots,{\bf{a_k}}\} \subseteq \mathbb{F}^n$ be a finite 
set and let
$\vec{m}:=(m_1,\dots,m_k)\in\mathcal{M}^k$. If there is 
${\bf{a}}\in \mathbb{F}^n\setminus \Omega$ such that either 
$$I^{\vec{m}}(\Omega) \supsetneq I^{\vec{m},o}(\Omega\cup\{{\bf a}\}) \quad \left(I^{\vec{m}}_L(\Omega) \supsetneq I^{\vec{m},o}_L(\Omega\cup\{{\bf a}\})\right)$$
or $I^{\vec{m},m'}(\Omega\cup\{{\bf a}\}) \supsetneq I^{\vec{m},x_jm'}(\Omega\cup\{{\bf a}\})\ \ \left(I^{\vec{m},m'}_L(\Omega\cup\{{\bf a}\}) \supsetneq I^{\vec{m},m'x_j}_L(\Omega\cup\{{\bf a}\})\right)$ for some $m'\in\mathcal{M}$, possibly $m'=o$, then there exists either
$$F\in I^{\vec{m}}(\Omega) \setminus I^{\vec{m},o}(\Omega\cup\{{\bf{a}}\}) \quad \left(I^{\vec{m}}_L(\Omega) \setminus I^{\vec{m},o}_L(\Omega\cup\{\bf{a}\})\right)$$ 
or $F\in I^{\vec{m},m'}(\Omega\cup\{{\bf a}\}) \setminus I^{\vec{m},x_jm'}(\Omega\cup\{{\bf{a}}\}) \ \left(I^{\vec{m},m'}_L(\Omega\cup\{{\bf a}\}) \setminus I^{\vec{m},m'x_j}_L(\Omega\cup\{\bf{a}\})\right)$,
such that either $\deg (F) \leq \sum_{i=1}^{k}(\deg(m_i)+1)\ ,$ or \ $\deg (F) \leq \sum_{i=1}^{k}\left(\deg(m_i)+1\right)+\left(\deg(m')+1 \right)\ $, respectively.
\end{lemma}

\begin{proof}
Since either $I^{\vec{m}}(\Omega) \supsetneq I^{\vec{m},o}(\Omega\cup\{{\bf{a}}\})$, or $I^{\vec{m},m'}(\Omega\cup\{{\bf a}\}) \supsetneq I^{\vec{m},x_jm'}(\Omega\cup\{{\bf a}\})$, take either
$F\in I^{\vec{m}}(\Omega) \setminus I^{\vec{m},o}(\Omega\cup\{{\bf{a}}\})$, or $F\in I^{\vec{m},m'}(\Omega\cup\{{\bf a}\}) \setminus I^{\vec{m},x_jm'}(\Omega\cup\{{\bf{a}}\})$, such that $LM(F)$ is minimum possible with
respect to $\prec$, where $\prec$ denotes any monomial order of $\mathcal{M}$ preserving degrees. If we suppose that either $\deg(F) \geq N+ 1$, or $\deg(F) \geq N+(\deg (m')+ 1)+1$, where $N:=\sum_{i=1}^{k}(\deg(m_i)+1)$, then either $\deg(LM(F)) \geq N + 1$, or $\deg(LM(F)) \geq N +(\deg (m')+ 1)+1$,
by the choice of $\prec$. Then, by applying Lemma \ref{3.1.8}, we can construct a monic skew polynomial $G\in I^{\vec{m}}(\Omega)\ \left(G\in  I^{\vec{m},m'}(\Omega\cup\{{\bf a}\})\right)$ such that $LM(F)=m''\cdot LM(G)$ for some $m''\in\mathcal{M}$ with $\deg (m'')\geq 1$. If either $G({\bf a})\neq 0$, or $\Delta_{\bf a}^{x_jm'}G({\bf a})\neq 0$, then we deduce that either $G\in I^{\vec{m}}(\Omega) \setminus I^{\vec{m},o}(\Omega\cup\{\bf{a}\})$, or $G\in I^{\vec{m},m'}(\Omega) \setminus I^{\vec{m},x_jm'}(\Omega\cup\{\bf{a}\})$, a contradiction because $\deg (F)>\deg (G)$. Suppose now that either $G({\bf a})=0$, or $\Delta_{\bf a}^{x_jm'}G({\bf a})=0$. Then there exists $\alpha \in \mathbb{F}$ such that $H:= F - \alpha m' \cdot G$ satisfies $LM(H)\prec LM(F)$. Now, by the definition of $H$, it holds that either $H\in I^{\vec{m}}(\Omega) \setminus I^{\vec{m},o}(\Omega\cup\{\bf{a}\})$, or $H\in I^{\vec{m},m'}(\Omega\cup\{{\bf a}\}) \setminus I^{\vec{m},x_jm'}(\Omega\cup\{\bf{a}\})$,
which is absurd by the minimality of $LM(F)$. Therefore we deduce that either $\deg(F)\leq N$, or $\deg(F)\leq N+(\deg(m')+1)$, respectively.
\end{proof}

By using the notion of left (right) ideals as in Definition \ref{def-ideals}, we can introduce a left (right) derivative polynomial independence notion, as follows.

\begin{definition}\label{pd-independent}
For $k\in\mathbb{Z}_{\geq 1}$, let $\Omega=\{{\bf{a_1}}, \dots ,{\bf{a_k}}\} \subseteq \mathbb{F}^n$ be a finite set and consider $m'=x_{i_t}\cdots x_{i_2}x_{i_1}\in\mathcal{M}$ and $\vec{m}=(m_1,\dots,m_k)\in \mathcal{M}^k$ with either $m_j=x_{j_{s(j)}}\cdots x_{j_2}x_{j_1}$, or $m_j=o$, for all $j=1,\dots , k$. 
\begin{enumerate}
\item ${\bf{a}}\in \mathbb{F}^{n}\setminus \Omega$ is \textbf{left (right) Derivative Polynomial independent} (or, simply, \textbf{DP-independent}) \textbf{of type $(\vec{m},m')$  from $\Omega$} if and only if
$$I^{\vec{m}}(\Omega) \supsetneq I^{\vec{m},o}(\Omega\cup \{{\bf{a}}\})\supsetneq I^{\vec{m},x_{i_1}}(\Omega\cup \{{\bf{a}}\})\supsetneq I^{\vec{m},x_{i_2}x_{i_1}}(\Omega\cup \{{\bf{a}}\})\supsetneq\dots \supsetneq I^{\vec{m},m'}(\Omega\cup \{{\bf{a}}\})
$$ 
$$\left( I^{\vec{m}}_L(\Omega) \supsetneq I^{\vec{m},o}_L(\Omega\cup \{{\bf{a}}\})\supsetneq I^{\vec{m},x_{i_t}}_L(\Omega\cup \{{\bf{a}}\})\supsetneq \dots \supsetneq I^{\vec{m},m'}(\Omega\cup \{{\bf{a}}\})\right) \ ;$$
\item $\Omega$ is \textbf{left (right) DP-independent of type $\vec{m}=(m_1,\dots,m_k)\in\mathcal{M}^k$} if and only if
${\bf a_i}$ is left (right) DP-independent of type $(\vec{m}_i,m_i)$ from $\Omega_{(i)}$ for all $i=1,\dots,k$, where $\Omega_{(j)}:=\Omega\setminus\{{\bf {a}_j}\}$ and $\vec{m}_j:=(m_1,\dots,m_{j-1},m_{j+1},\dots,m_k)\in\mathcal{M}^{k-1}$ for each $j\in\{1,\dots,k\}$; 
\item $\Omega$ is \textbf{left (right) P-independent} if and only if $\Omega$ is left (right) DP-independent of type $(o,\dots,o)\in\mathcal{M}^k$.   
\end{enumerate}
\end{definition}

\begin{remark}\label{rem1}
An element ${\bf a}\in\mathbb{F}^n$ is P-independent from $\Omega\subseteq\mathbb{F}^n$ (see \cite[Definition 23]{evaluationandinterpolation}) if and only if ${\bf{a}}$ is left DP-independent of type $(\vec{0},o)$  from $\Omega$, where $\vec{0}:=(o,\dots,o)\in\mathcal{M}^k$. Moreover, the definition of $\Omega:=\{{\bf a_1},\dots,{\bf a_k} \}\subseteq\mathbb{F}^n$ as a P-independent set given in \cite[Definition 23]{evaluationandinterpolation} is equivalent to require that $I(\Omega_{(j)}) \supsetneq I(\Omega)$ for all $j=1,\dots,k$. Thus, in our setting, a (finite) P-independent set $\Omega$ as above is simply a left P-independent set (that is, a left DP-independent set of type $\vec{0}\in\mathcal{M}^k$). 
\end{remark}

\medskip

From Definition \ref{pd-independent}, we can deduce the following property of PD-independent sets.

\begin{proposition}\label{prop1}
If $\Omega= \{{\bf{a_1}},\dots,{\bf{a_k}}\} \subseteq \mathbb{F}^n$ is left (right) DP-independent of type $(m_1,\dots,m_k)$, then any $W=\{{\bf{a_{j_1}}},\dots,{\bf{a_{j_s}}}\}\subseteq \Omega$ is left (right) DP-independent of type $(m_{j_1}, \dots , m_{j_s})$, where $j_i\in\{1,\dots,k\}$. In particular, every subset of a left {(right)} P-independent set is left (right) P-independent.
\end{proposition}

\begin{proof}
Let ${t}\in\{ {{{j_1}}},\dots,{{{j_s}}} \}$, $\vec{m}:=(m_1,\dots,m_k)$ and $\vec{m}':=(m_{j_1}, \dots , m_{j_s})$. Moreover, write $m_t:=x_{t_s}\cdots x_{t_2}x_{t_1}$. Since for every $j=1,\dots,s-1$, we have
$$I^{\vec{m}_{t}}\left(\Omega_{(t)}\right) \supsetneq I^{\vec{m}_t,o}(\Omega_{(t)}\cup\{{{\bf a}_{t}}\})\supsetneq \cdots\supsetneq I^{\vec{m}_t,x_{t_j}\cdots x_{t_1}}(\Omega_{(t)}\cup\{{{\bf a}_{t}}\})\supsetneq  I^{\vec{m}_t,x_{t_{j+1}}\cdots x_{t_1}}(\Omega_{(t)}\cup\{{{\bf a}_{t}}\}) \ ,$$
it follows that there exists $F\in I^{\vec{m}_{{\bf t}},x_{t_{j}}\cdots x_{t_1}}\left(\Omega_{(t)}\cup\{{{\bf a}_{t}}\}\right) \subseteq I^{\vec{m}'_{t},x_{t_{j}}\cdots x_{t_1}}\left(W_{(t)}\cup\{{{\bf a}_{t}}\}\right)$ such that $\Delta_{{{\bf a}_t}}^{x_{t_{j+1}}\cdots x_{t_1}}F({{\bf a}_t})\neq 0$. Hence $I^{\vec{m}'_{t},x_{t_{j}}\cdots x_{t_1}}\left(W_{({t})}\cup\{{{\bf a}_{t}}\}\right) \supsetneq I^{\vec{m}'_t,x_{t_{j+1}}\cdots x_{t_1}}(W_{({t})}\cup\{{{\bf a}_{t}}\})$ for every  ${t}\in\{ {{{j_1}}},\dots,{{{j_s}}} \}$ and for any $j=1,\dots,s-1$, i.e. $W=\{{{\bf a}_{j_1}},\dots,{{\bf a}_{j_s}}\}\subseteq \Omega$ is left DP-independent of type $\vec{m}'$. Finally, note that the proof of the right version is quite similar and that the last part of the statement follows from Definition \ref{pd-independent} (3).
\end{proof}

\medskip

\noindent The next result will be useful
to perform a skew Hermite-type interpolation recursively in $\mathcal{R}$ and it extends the equivalence between the points $1.$ and $3.$ of \cite[Proposition 25]{evaluationandinterpolation}.

\begin{proposition}\label{prop 3.1.7}
Let $\Omega=\{{\bf{a_1}},\dots,{\bf{a_k}}\} \subseteq \mathbb{F}^n$ be a finite set and let $m_1,\dots,m_k\in \mathcal{M}$. Then the following conditions are equivalent:
\begin{itemize}
    \item[$1)$] $\Omega$ is left (right) DP-independent of type $(m_1,\dots,m_k)$;
    \item[$2)$] for any ordering ${\bf{b_1}},\dots,{\bf{b_k}}$ of the elements in $\Omega$ and for any $i =
1,\dots,k-1$, it holds that $\bf{b_{i+1}}$ is left (right) DP-independent of type $(m'_{1},\dots,m'_{i},m'_{i+1})$ from $\Omega_i:= \{{\bf{b_1}},\dots,{\bf{b_i}}\}$,
where $m'_j=m_{k(j)}$ with ${\bf b_j}={\bf a_{k(j)}}$ for $j=1,\dots,i$.
\end{itemize}
\end{proposition}

\begin{proof}
$1)\Rightarrow 2)$ Suppose that ${\bf{b_{i+1}}}$ is not left DP-independent of type $(m'_{1},\dots,m'_{i+1})$ from $\Omega_i$
for some $i$ and a given ordering ${\bf{b_1}},\dots,{\bf{b_k}}$
of $\Omega$. Then, from Definition \ref{pd-independent} it follows that either $I^{(m'_1, \dots , m'_i)}(\Omega_i)=I^{(m'_1, \dots , m'_i,o)}(\Omega_i\cup\{{\bf b_{i+1}} \})$, or 
$$I^{(m'_1, \dots , m'_i,x_{t_{j-1}}\cdots x_{t_1})}(\Omega_i\cup\{{\bf b_{i+1}} \})=I^{(m'_1, \dots , m'_i,x_{t_j}\cdots x_{t_1})}(\Omega_i\cup\{{\bf b_{i+1}} \})$$ for some $j\in\{1,\dots,s \}$, where $m'_{i+1}=x_{t_s}\cdots x_{t_2}x_{t_1}$ and $x_{t_0}:=o$, but this contradicts Proposition \ref{prop1} by considering $W:=\{{\bf{b_1}},\dots,{\bf{b_i}},{\bf{b_{i+1}}}\}\subseteq\Omega$.

$2)\Rightarrow 1)$ Assume that $\Omega$ is not left DP-independent of type $\vec{m}:=(m_1, \dots , m_k)$. Thus by Definition \ref{pd-independent} there exists ${\bf a_i}\in \Omega$ such that ${\bf a_i}$ is not left PD-independent of type $(m_1,\dots,m_{i-1},m_{i+1},\dots,m_k,m_i)$ from $\Omega_{({\bf a_i})}:=\Omega\setminus \{{\bf a_i}\}$. By ordering the $k$ elements in
$\Omega$ in such a way that ${\bf{b_{k}}} = {\bf a_i}$, it follows that ${\bf b_k}$ is not left DP-independent of type $(m'_{1},\dots,m'_{k-1},m'_{k})$ from $\Omega_{k-1}$, but this contradicts $2)$.
\end{proof}

Finally, let us give here some conditions equivalent to Definition \ref{pd-independent} (2). Before to do this, we need further notation and 
results.

\begin{definition}
Let $\Omega'=\{{\bf{b_1}}, \dots ,{\bf{b_h}}\} \subseteq \mathbb{F}^n$ be a finite set and consider $\vec{m}'=(m'_1,\dots,m'_h)\in \mathcal{M}^h$.
We denote by $V_{\leq d'}^{\vec{m}'}\left(\Omega'\right)\ \left(V_{d'}^{\vec{m}'}\left(\Omega'\right) \right)$ $\left(\mathrm{or}\ V_{L,{\leq d'}}^{\vec{m}'}\left(\Omega'\right)\ \left(V_{L,{d'}}^{\vec{m}'}\left(\Omega'\right) \right)\right)$ the finite generated left (right) $\mathbb{F}$-module of all skew polynomials $F\in I^{\vec{m}'}\left(\Omega'\right)$ (or $F\in I_L^{\vec{m}'}\left(\Omega'\right)$) such that $\deg(F)\leq d'\ \left(\deg(F)=d'\right)$.
\end{definition}

\begin{proposition}\label{Prop}
Let $\Omega'=\{{\bf{b_1}}, \dots ,{\bf{b_h}}\} \subseteq \mathbb{F}^n$ be a finite set and consider $\vec{m}'=(m'_1,\dots,m'_h)\in \mathcal{M}^h$ and $d'\in\mathbb{Z}_{>0}$. Then we have the following properties:

\smallskip

\begin{enumerate}
\item[$a)$] $V_{\leq d'}^{\vec{m}'}\left(\Omega'\right)=V_{d'}^{\vec{m}'}\left(\Omega'\right)\oplus V_{\leq d'-1}^{\vec{m}'}\left(\Omega'\right)\ \left(V_{L,\leq d'}^{\vec{m}'}\left(\Omega'\right)=V_{L,d'}^{\vec{m}'}\left(\Omega'\right)\oplus V_{L,\leq d'-1}^{\vec{m}'}\left(\Omega'\right)\right) \ ;$

\smallskip

\item[$b)$] $\dim_{\mathbb{F}} V_{\leq d'}^{\vec{m}'}\left(\Omega'\right)=\dim_{\mathbb{F}} V_{d'}^{\vec{m}'}\left(\Omega'\right)+\dim_{\mathbb{F}} V_{\leq d'-1}^{\vec{m}'}\left(\Omega'\right)$ 

\noindent $\left( \dim_{\mathbb{F}} V_{L,\leq d'}^{\vec{m}'}\left(\Omega'\right)=\dim_{\mathbb{F}} V_{L,d'}^{\vec{m}'}\left(\Omega'\right)+\dim_{\mathbb{F}} V_{L,\leq d'-1}^{\vec{m}'}\left(\Omega'\right) \right)$ ;

\smallskip

\item[$c)$] $n\cdot \dim_{\mathbb{F}} V_{d'-1}^{\vec{m}'}\left(\Omega'\right) \leq\dim_{\mathbb{F}} V_{d'}^{\vec{m}'}\left(\Omega'\right)\leq n^{d'}$

\noindent $\left( n\cdot \dim_{\mathbb{F}} V_{L,d'-1}^{\vec{m}'}\left(\Omega'\right) \leq\dim_{\mathbb{F}} V_{L,d'}^{\vec{m}'}\left(\Omega'\right)\leq n^{d'} \right)$ ;

\smallskip

\item[$d)$] consider
$m'\in\mathcal{M}$, ${\bf a}\in\mathbb{F}^n\setminus\Omega'$ and define $N:=\sum_{i=1}^{h}\left(\deg(m'_i)+1 \right)$; then

\smallskip

\begin{enumerate}
\item[$(i)$]
$0\leq\dim_{\mathbb{F}} V_{\leq N+\deg(m')+1}^{(\vec{m}',m')}\big(\Omega'\cup \{{\bf a} \}\big)-\dim_{\mathbb{F}} V_{\leq N+\deg(m')+1}^{(\vec{m}',x_im')}\big(\Omega'\cup \{{\bf a} \}\big)  \leq 1$

{\small
\noindent $\left(0\leq\dim_{\mathbb{F}} V_{L,\leq N+\deg(m')+1}^{(\vec{m}',m')}\big(\Omega'\cup \{{\bf a} \}\big)-\dim_{\mathbb{F}} V_{L,\leq N+\deg(m')+1}^{(\vec{m}',x_im')}\big(\Omega'\cup \{{\bf a} \}\big)  \leq 1\right)$~;
}

\smallskip

\item[$(ii)$] if we have
$$V^{(\vec{m}',m')}_{\leq d'}\left(\Omega'\cup\{{\bf a} \} \right)\supsetneq V^{(\vec{m}',x_im')}_{\leq d'}\left(\Omega'\cup\{{\bf a} \} \right)$$ 
$$\left( V^{(\vec{m}',m')}_{L,\leq d'}\left(\Omega'\cup\{{\bf a} \} \right)\supsetneq V^{(\vec{m}',x_im')}_{L,\leq d'}\left(\Omega'\cup\{{\bf a} \} \right) \right)\ ,$$ then for every $d\geq d'$ we get
$$V^{(\vec{m}',m')}_{\leq d}\left(\Omega'\cup\{{\bf a} \} \right)\supsetneq V^{(\vec{m}',x_im')}_{\leq d}\left(\Omega'\cup\{{\bf a} \} \right)$$
$$\left(V^{(\vec{m}',m')}_{L,\leq d}\left(\Omega'\cup\{{\bf a} \} \right)\supsetneq V^{(\vec{m}',x_im')}_{L,\leq d}\left(\Omega'\cup\{{\bf a} \} \right)\right)\ ;$$  

\smallskip

\item[$(iii)$] We have 
$$I^{(\vec{m}',m')}\left(\Omega'\cup\{{\bf a} \} \right)\supsetneq I^{(\vec{m}',x_im')}\left(\Omega'\cup\{{\bf a} \} \right)$$
$$\left(I_L^{(\vec{m}',m')}\left(\Omega'\cup\{{\bf a} \} \right)\supsetneq I_L^{(\vec{m}',x_im')}\left(\Omega'\cup\{{\bf a} \} \right)\right)$$ if and only if
$$V^{(\vec{m}',m')}_{\leq N+\deg(m')+1}\left(\Omega'\cup\{{\bf a} \} \right)\supsetneq V^{(\vec{m}',x_im')}_{\leq N+\deg(m')+1}\left(\Omega'\cup\{{\bf a} \} \right)$$
$$\left(V^{(\vec{m}',m')}_{L,\leq N+\deg(m')+1}\left(\Omega'\cup\{{\bf a} \} \right)\supsetneq V^{(\vec{m}',x_im')}_{L,\leq N+\deg(m')+1}\left(\Omega'\cup\{{\bf a} \} \right)\right)\ ,$$
and $I^{(\vec{m}')}\left(\Omega'\right)\supsetneq I^{(\vec{m}',o)}\left(\Omega'\cup\{{\bf a} \} \right)$ $\left(I_L^{(\vec{m}')}\left(\Omega'\right)\supsetneq I_L^{(\vec{m}',o)}\left(\Omega'\cup\{{\bf a} \} \right)\right)$ if and only if \
$V^{(\vec{m}')}_{\leq N}\left(\Omega'\right)\supsetneq V^{(\vec{m}',o)}_{\leq N}\left(\Omega'\cup\{{\bf a} \} \right)$ $\left(V^{(\vec{m}')}_{L,\leq N}\left(\Omega'\right)\supsetneq V^{(\vec{m}',o)}_{L,\leq N}\left(\Omega'\cup\{{\bf a} \} \right)\right)\ ;$

\smallskip

\item[$(iv)$] for every $j=1,\dots , h$, we have
$$\dim_{\mathbb{F}} V^{\vec{m}'_j}_{\leq N}\left(\Omega'_{(j)}\right)=n^{N_j+1}\cdot\left(\sum_{t=0}^{\deg (m'_j)}n^t\right)+\dim_{\mathbb{F}} V^{\vec{m}'_j}_{\leq N_j}\left(\Omega'_{(j)}\right)$$
$$\qquad \left(\dim_{\mathbb{F}} V^{\vec{m}'_j}_{\leq N}\left(\Omega'_{(j)}\right)=n^{N_j+1}\cdot\left(\sum_{t=0}^{\deg (m'_j)}n^t\right)+\dim_{\mathbb{F}} V^{\vec{m}'_j}_{\leq N_j}\left(\Omega'_{(j)}\right)\right)\ ,$$

where $N_j:=\sum_{i=1, i\neq j}^{k}\left[\deg(m'_i)+1\right]$ . 

\end{enumerate}
\end{enumerate}

\end{proposition}

\begin{proof}
For simplicity, we will give here the proof of 
$(i)$ to $(iv)$ only for the right cases, leaving to the reader the other cases of the statement written in parentheses. 

\smallskip

The above properties $a), b)$ and $c)$ follow from the definitions of $V_{\leq d'}^{\vec{m}'}\left(\Omega'\right)$ and $V_{d'}^{\vec{m}'}\left(\Omega'\right)$.
To prove $d)(i)$, firstly note that 
$$V_{\leq N+\deg(m')+1}^{(\vec{m}',m')}\big(\Omega'\cup \{{\bf a} \}\big)\supseteq V_{\leq N+\deg(m')+1}^{(\vec{m}',x_im')}\big(\Omega'\cup \{{\bf a} \}\big)\ ,$$ which clearly gives the first inequality in $d)(i)$. Secondly, assume that 
$$V_{\leq N+\deg(m')+1}^{(\vec{m}',m')}\big(\Omega'\cup \{{\bf a} \}\big)\supsetneq V_{\leq N+\deg(m')+1}^{(\vec{m}',x_im')}\big(\Omega'\cup \{{\bf a} \}\big)\ .$$
Suppose that there are $F_1,F_2\in V_{\leq N+\deg(m')+1}^{(\vec{m}',m')}\big(\Omega'\cup \{{\bf a} \}\big)\setminus V_{\leq N+\deg(m')+1}^{(\vec{m}',x_im')}\big(\Omega'\cup \{{\bf a} \}\big)$ such that $\Delta_{{\bf a}}^{x_im'}F_t({\bf a})=\mu_t\in\mathbb{F}^*$ for $t=1,2$. By defining $G:=\mu_1^{-1}F_1+\mu_2^{-1}F_2$, we see that $G\in V_{\leq N+\deg(m')+1}^{(\vec{m}',x_im')}\big(\Omega'\cup \{{\bf a} \}\big)$. This shows that
$$F_2\in \langle F_1 \rangle\oplus V_{\leq N+\deg(m')+1}^{(\vec{m}',x_im')}\big(\Omega'\cup \{{\bf a} \}\big)\ ,$$
i.e., $V_{\leq N+\deg(m')+1}^{(\vec{m}',m')}\big(\Omega'\cup \{{\bf a} \}\big)=\langle F_1 \rangle\oplus V_{\leq N+\deg(m')+1}^{(\vec{m}',x_im')}\big(\Omega'\cup \{{\bf a} \}\big)$. Therefore, we get
$$\dim_{\mathbb{F}} V_{\leq N+\deg(m')+1}^{(\vec{m}',m')}\big(\Omega'\cup \{{\bf a} \}\big)-\dim_{\mathbb{F}} V_{\leq N+\deg(m')+1}^{(\vec{m}',x_im')}\big(\Omega'\cup \{{\bf a} \}\big)  \leq 1 \ .$$
The property $d)(ii)$ follows from the fact that 
$V^{(\vec{m}',m')}_{\leq d'}\left(\Omega'\cup\{{\bf a} \} \right)\subseteq V^{(\vec{m}',m')}_{\leq d}\left(\Omega'\cup\{{\bf a} \} \right)$ for every $d\geq d'$. Now, let us prove $d)(iii)$. Observe that the ``if'' parts are immediate. So, assume first that $I^{(\vec{m}')}\left(\Omega'\right)\supsetneq I^{(\vec{m}',o)}\left(\Omega'\cup\{{\bf a} \} \right)$. By contradiction, suppose that 
$$V^{(\vec{m}')}_{\leq N}\left(\Omega'\right)=V^{(\vec{m}',o)}_{\leq N}\left(\Omega'\cup\{{\bf a} \} \right)\ .$$ Then from $d)(ii)$ it follows that $V^{(\vec{m}')}_{\leq d}\left(\Omega'\right)=V^{(\vec{m}',o)}_{\leq d}\left(\Omega'\cup\{{\bf a} \} \right)$ for every $d\leq N$. Moreover, if
$V^{(\vec{m}')}_{\leq d}\left(\Omega'\right)\supsetneq V^{(\vec{m}',o)}_{\leq d}\left(\Omega'\cup\{{\bf a} \} \right)$ for some $d>N$, then there exists $d'>N$ such that $V^{(\vec{m}')}_{\leq d'}\left(\Omega'\right)\supsetneq V^{(\vec{m}',o)}_{\leq d'}\left(\Omega'\cup\{{\bf a} \} \right)$ and $V^{(\vec{m}')}_{\leq d'-1}\left(\Omega'\right)=V^{(\vec{m}',o)}_{\leq d'-1}\left(\Omega'\cup\{{\bf a} \} \right)$. Thus, there is $F\in V^{(\vec{m}')}_{d'}\left(\Omega'\right)$ such that $F({\bf a})\neq 0$ and $LM(F)$ is minimum possible with respect to a fixed ordering $\prec$ of $\mathcal{M}$ preserving degrees. 
By Lemma \ref{3.1.8}, we know that there exists $G\in V^{(\vec{m}')}_{\leq N}\left(\Omega'\right)=V^{(\vec{m}',o)}_{\leq N}\left(\Omega'\cup\{{\bf a} \} \right)$ such that $LM(F)=\alpha m'LM(G)$ for some $\alpha\in\mathbb{F}$ and $m'\in\mathcal{M}$. Defining $H:=F-\alpha m' G$, we get $H\in V^{(\vec{m}')}_{d'}\left(\Omega'\right)$, $H({\bf a})=F({\bf a})-(\alpha m' G)({\bf a})=F({\bf a})\neq 0$ and $LM(H)\prec LM(F)$, but this is a contradiction by the minimality of $LM(F)$. Therefore, we have $V^{(\vec{m}')}_{\leq d}\left(\Omega'\right)=V^{(\vec{m}',o)}_{\leq d}\left(\Omega\cup\{{\bf a} \} \right)$ for every $d$, but this is impossible because $I^{(\vec{m}')}\left(\Omega'\right)\supsetneq I^{(\vec{m}',o)}\left(\Omega'\cup\{{\bf a} \} \right)$. Hence $V^{(\vec{m}')}_{\leq N}\left(\Omega'\right)\supsetneq V^{(\vec{m}',o)}_{\leq N}\left(\Omega'\cup\{{\bf a} \} \right)$. Let us note that by similar arguments as above, we can prove that if $I^{(\vec{m}',m')}\left(\Omega'\cup\{{\bf a} \} \right)\supsetneq I^{(\vec{m}',x_im')}\left(\Omega'\cup\{{\bf a} \} \right)$, then $V^{(\vec{m}',m')}_{\leq N+\deg(m')+1}\left(\Omega'\cup\{{\bf a} \} \right)\supsetneq V^{(\vec{m}',x_im')}_{\leq N+\deg(m)+1}\left(\Omega'\cup\{{\bf a} \} \right)$. 
Finally, we prove $d)(iv)$. 
From Lemma \ref{3.1.8}, it follows that 
$\dim_{\mathbb{F}} V^{\vec{m}'_j}_{N_j}\left(\Omega'_{(j)}\right)=n^{N_j}$
and, by iterating $c)$, we get $\dim_{\mathbb{F}} V^{\vec{m}'_j}_{N_j+h}\left(\Omega'_{(j)}\right)=n^{N_j+h}$ for any $h=0,\dots,\deg(m'_j)+1$. Moreover, 
by iterating $b)$, we obtain that
$$\dim_{\mathbb{F}} V^{\vec{m}'_j}_{\leq N}\left(\Omega'_{(j)}\right)=\dim_{\mathbb{F}} V^{\vec{m}'_j}_{\leq N_j+\deg(m'_j)+1}\left(\Omega'_{(j)}\right)=\dim_{\mathbb{F}} V^{\vec{m}'_j}_{N_j+\deg(m'_j)+1}\left(\Omega'_{(j)}\right)+$$
$$+\dim_{\mathbb{F}} V^{\vec{m}'_j}_{\leq N_j+\deg(m'_j)}\left(\Omega'_{(j)}\right)= n^{N_j+\deg(m'_j)+1}+\dim_{\mathbb{F}} V^{\vec{m}'_j}_{\leq N_j+\deg(m'_j)}\left(\Omega'_{(j)}\right)=\dots =$$
$$=n^{N_j+\deg(m'_j)+1}+\dots +n^{N_j+1}+\dim_{\mathbb{F}} V^{\vec{m}'_j}_{\leq N_j}\left(\Omega'_{(j)}\right)=n^{N_j+1}\cdot\left(\sum_{t=0}^{\deg (m'_j)}n^t\right)+\dim_{\mathbb{F}} V^{\vec{m}'_j}_{\leq N_j}\left(\Omega'_{(j)}\right)\ ,$$
that is,
$\dim_{\mathbb{F}} V^{\vec{m}'_j}_{\leq N}\left(\Omega'_{(j)}\right)=n^{N_j+1}\cdot\left(\sum_{t=0}^{\deg(m'_j)}n^t\right)+\dim_{\mathbb{F}} V^{\vec{m}'_j}_{\leq N_j}\left(\Omega'_{(j)}\right)\ .$
\end{proof}

The following result shows some conditions equivalent to Definition \ref{pd-independent} (2).

\begin{theorem}
Let $\Omega=\{{\bf{a_1}}, \dots ,{\bf{a_k}}\} \subseteq \mathbb{F}^n$ be a finite set and consider 
$\vec{m}=(m_1,\dots,m_k)$ in $\mathcal{M}^k$ with either $m_j=x_{j_{s(j)}}\cdots x_{j_2}x_{j_1}$, or $m_j=o$, for all $j=1,\dots , k$. Define $x_{j_0}:=o$, for every $j=1,\dots,k$, and $N:=\sum_{i=1}^{k}\left[\deg(m_i)+1\right]$.
Then the following are equivalent:
\begin{enumerate}
\item $\Omega$ is left (right) DP-independent of type $\vec{m}=(m_1,\dots,m_k)$ ;
\item for every $j=1,\dots,k$, we have 
$${\bf a_j}\notin Z^o\left(I^{\vec{m}_j}(\Omega_{(j)})\right)\cup \left[\bigcup_{i=1}^{s(j)}Z^{x_{j_i}\dots x_{j_1}}\Big(I^{\vec{m}_j,x_{j_{i-1}}\dots x_{j_1}}\left(\Omega_{(j)}\cup \{{\bf a_j}\}\right)\Big)\right]$$
$$\left({\bf a_j}\notin Z_L^o\left(I_L^{\vec{m}_j}(\Omega_{(j)})\right)\cup \left[\bigcup_{i=1}^{s(j)}Z_L^{x_{j_i}\dots x_{j_1}}\Big(I_L^{\vec{m}_j,x_{j_{i-1}}\dots x_{j_1}}\left(\Omega_{(j)}\cup \{{\bf a_j}\}\right)\Big)\right]\right) \ ;$$
\item for every $j=1,\dots,k$, we get
$$\dim_{\mathbb{F}} V^{\vec{m}_j}_{\leq N}\left(\Omega_{(j)}\right)=\dim_{\mathbb{F}} V^{\vec{m}_j,m_j}_{\leq N}\left(\Omega_{(j)}\cup \{ {\bf a_j} \}\right)+\deg(m_j)+1$$
$$\left(\dim_{\mathbb{F}} V^{\vec{m}_j}_{L,\leq N}\left(\Omega_{(j)}\right)=\dim_{\mathbb{F}} V^{\vec{m}_j,m_j}_{L,\leq N}\left(\Omega_{(j)}\cup \{ {\bf a_j} \}\right)+\deg(m_j)+1\right)\ ;$$
\item for every $j=1,\dots,k$, there exists $F\in I^{\vec{m}_j}\left(\Omega_{(j)}\right)$ $\left(F\in I_L^{\vec{m}_j}\left(\Omega_{(j)}\right)\right)$ such that $F({\bf a_j})\neq 0$ $\left(F_L({\bf a_j})\neq 0\right)$ and $\Delta_{{\bf a_j}}^{x_{j_{i}}\cdots x_{j_1}}F({\bf a_j})=0$ $\left(\left(\Delta_{{\bf a_j},L}^{x_{j_{i}}\cdots x_{j_1}}F\right)_L({\bf a_j})=0\right)$ for every $i=1,\dots,s(j)$ .
\end{enumerate}
\end{theorem}

\begin{proof}
For simplicity, we will give here only the proofs for the left cases.

The equivalence between \textit{(1)} and $\textit{(2)}$ follows from Definitions \ref{def-ideals} and \ref{def-zeros}. 

\medskip

\noindent \textit{(1)} $\Rightarrow$ \textit{(4)}. For every $j=1,\dots,k$, we have 
$$I^{\vec{m}_j}(\Omega_{(j)}) \supsetneq I^{\vec{m}_j,o}(\Omega_{(j)}\cup \{{\bf{a}_j}\})\supsetneq I^{\vec{m}_j,x_{j_1}}(\Omega_{(j)}\cup \{{\bf{a}_j}\})$$ and
$I^{\vec{m}_j,x_{j_i}\cdots x_{j_1}}(\Omega_{(j)}\cup \{{\bf{a}_j}\})\supsetneq I^{\vec{m}_j,x_{j_{i+1}}x_{j_i}\cdots x_{j_1}}(\Omega_{(j)}\cup \{{\bf{a}_j}\})$ for any $i=1,\dots,s(j)-1$. Then, by Lemma \ref{3.1.9} we deduce that there exist $F_0\in I^{\vec{m}_j}(\Omega_{(j)}) \setminus I^{\vec{m}_j,o}(\Omega_{(j)}\cup \{{\bf{a}_j}\})$, $F_1\in I^{\vec{m}_j,o}(\Omega_{(j)}\cup \{{\bf{a}_j}\})\setminus I^{\vec{m}_j,x_{j_1}}(\Omega_{(j)}\cup \{{\bf{a}_j}\})$
and 
$$F_{i+1}\in I^{\vec{m}_j,x_{j_i}\cdots x_{j_1}}(\Omega_{(j)}\cup \{{\bf{a}_j}\})\setminus I^{\vec{m}_j,x_{j_{i+1}}x_{j_i}\cdots x_{j_1}}(\Omega_{(j)}\cup \{{\bf{a}_j}\})$$ for any $i=1,\dots,s(j)-1$, such that $\deg(F_t)\leq N_j+t$ for every $t=0,1,\dots, s(j)$. Thus, there exist some suitable $\lambda_i\in\mathbb{F}$ such that $F:=F_0+\lambda_1F_1+\dots+\lambda_{s(j)}F_{s(j)}\in I^{\vec{m}_j}\left(\Omega_{(j)}\right)$, $F({\bf a_j})\neq 0$ and $\Delta_{{\bf a_j}}^{x_{j_{i}}\cdots x_{j_1}}F({\bf a_j})=0$ for every $i=1,\dots,s(j)$. 

\medskip

\noindent \textit{(4)} $\Rightarrow$ \textit{(1)}. Consider $F\in I^{\vec{m}_j}\left(\Omega_{(j)}\right)$ such that $F({\bf a_j})\neq 0$ and $\Delta_{{\bf a_j}}^{x_{j_{i}}\cdots x_{j_1}}F({\bf a_j})=0$ for every $i=1,\dots,s(j)$. Then we get $I^{\vec{m}_j}(\Omega_{(j)}) \supsetneq I^{\vec{m}_j,o}(\Omega_{(j)}\cup \{{\bf{a}_j}\})$. Moreover, multiplying $F$ by a suitable scalar, we can assume that $F({\bf a_j})=1$. So, define $F_1:=\left(x_{j_1}-({\bf a_j})_{j_1}\right)F$. Note that $F_1\in I^{\vec{m}_j,o}(\Omega_{(j)}\cup \{{\bf{a}_j}\})$. Furthermore, we have
$$\Delta_{{\bf a_j}}^{x_{j_1}}F_1=\left(x_{j_1}-({\bf a_j})_{j_1}\right)\Delta_{{\bf a_j}}^{x_{j_1}}F+\Delta_{{\bf a_j}}^{x_{j_1}}\big(\left(x_{j_1}-({\bf a_j})_{j_1}\right)F({\bf a_j}) \big)=\left(x_{j_1}-({\bf a_j})_{j_1}\right)\Delta_{{\bf a_j}}^{x_{j_1}}F+1\ .$$
Hence $\Delta_{{\bf a_j}}^{x_{j_1}}F_1({\bf a_j})=1\neq 0$ and this shows that $I^{\vec{m}_j,o}(\Omega_{(j)}\cup \{{\bf{a}_j}\})\supsetneq I^{\vec{m}_j,x_{j_1}}(\Omega_{(j)}\cup \{{\bf{a}_j}\})$. Now, define $F_2:=\left(x_{j_2}-({\bf a_j})_{j_2}\right)F_1$ and observe that 
$$\Delta_{{\bf a_j}}^{x_{j_1}}F_2=\left(x_{j_2}-({\bf a_j})_{j_2}\right)\Delta_{{\bf a_j}}^{x_{j_1}}F_1+\Delta_{{\bf a_j}}^{x_{j_1}}\big(\left(x_{j_2}-({\bf a_j})_{j_2}\right)F_1({\bf a_j}) \big)=\left(x_{j_2}-({\bf a_j})_{j_2}\right)\Delta_{{\bf a_j}}^{x_{j_1}}F_1\ ,$$
$$\Delta_{{\bf a_j}}^{x_{j_2}x_{j_1}}F_2=\left(x_{j_2}-({\bf a_j})_{j_2}\right)\Delta_{{\bf a_j}}^{x_{j_2}x_{j_1}}F_1+1=\left(x_{j_2}-({\bf a_j})_{j_2}\right)\left(x_{j_1}-({\bf a_j})_{j_1}\right)\Delta_{{\bf a_j}}^{x_{j_2}x_{j_1}}F+$$
$$+\left(x_{j_2}-({\bf a_j})_{j_2}\right)\Delta_{{\bf a_j}}^{x_{j_2}}\big( \left(x_{j_1}-({\bf a_j})_{j_1}\right)\Delta_{{\bf a_j}}^{x_{j_1}}F({\bf a_j})\big)+1 = $$
$$=\left(x_{j_2}-({\bf a_j})_{j_2}\right)\left(x_{j_1}-({\bf a_j})_{j_1}\right)\Delta_{{\bf a_j}}^{x_{j_2}x_{j_1}}F+\left(x_{j_2}-({\bf a_j})_{j_2}\right)\Delta_{{\bf a_j}}^{x_{j_2}}\big( \left(x_{j_1}-({\bf a_j})_{j_1}\right)\big)+1\ .$$
This gives $F_2\in I^{\vec{m}_j,x_{j_1}}(\Omega_{(j)}\cup \{{\bf{a}_j}\})$ with $\Delta_{{\bf a_j}}^{x_{j_2}x_{j_1}}F_2({\bf a_j})=1$, and this shows that 
$I^{\vec{m}_j,x_{j_1}}(\Omega_{(j)}\cup \{{\bf{a}_j}\})\supsetneq I^{\vec{m}_j,x_{j_2}x_{j_1}}(\Omega_{(j)}\cup \{{\bf{a}_j}\})$. By a recursive argument, one can see that \textit{(1)} holds.

\medskip

\noindent \textit{(1)} $\iff$ \textit{(3)}. 
By Proposition \ref{Prop} $d)(ii), (iii)$, we see that \textit{(1)} holds if and only if for every $j\in\{1,\dots,k\}$ we have
$$V^{\vec{m}_j}_{\leq N}\left(\Omega_{(j)}\right)\supsetneq V^{(\vec{m}_j,o)}_{\leq N}\left(\Omega_{(j)}\cup\{{\bf a_j} \} \right)
\supsetneq \cdots \supsetneq V^{(\vec{m}_j,m_j)}_{\leq N}\left(\Omega_{(j)}\cup\{{\bf a_j} \} \right)\ .$$
Thus, by Proposition \ref{Prop} $d)(i)$, we deduce that \textit{(1)} holds if and only if for every $j\in\{1,\dots,k\}$ we have
$$\dim_{\mathbb{F}} V^{\vec{m}_j}_{\leq N}\left(\Omega_{(j)}\right)=\dim_{\mathbb{F}} V^{\vec{m}_j,m_j}_{\leq N}\left(\Omega_{(j)}\cup \{ {\bf a_j} \}\right)+s(j)+1\ ,$$
i.e., $\dim_{\mathbb{F}} V^{\vec{m}_j}_{\leq N}\left(\Omega_{(j)}\right)=\dim_{\mathbb{F}} V^{\vec{m}_j,m_j}_{\leq N}\left(\Omega_{(j)}\cup \{ {\bf a_j} \}\right)+\deg(m_j)+1\ .$ 
\end{proof}

\medskip

The main result here is a Hermite-type interpolation theorem in $\mathcal{R}$ that generalizes the skew Lagrange interpolation given in \cite[Theorem 4]{evaluationandinterpolation} and \cite[Theorem 3]{zerosmartinez}, and it extends the cases $n=1$ given in \cite[Theorem 4.4]{Er1} and \cite[Corollary 41]{zerosmartinez}.

\begin{theorem}[\textbf{Skew Hermite-type interpolation}]\label{theorem 3.1.7} Let $\Omega=\{{\bf{a_1}},\dots,{\bf{a_k}}\} \subseteq \mathbb{F}^n$ be a finite set, $(m_1,\dots,m_k)\in \mathcal{M}^k$ and define $N:=\sum_{j=1}^k(\deg(m_j)+1)$. Then, the following conditions are equivalent:
\begin{itemize}
    \item[$1)$] $\Omega$ is left (right) DP-independent  of type $(m_1, \dots , m_k)$;
    \item[$2)$] there exist $F_i\in\mathcal{R}$ with $\deg(F_i)<N$ for $i=1,\dots,N$, such that the map $\psi: \langle F_1,\dots,F_N\rangle
    \to \mathbb{F}^{N}$ $\left(\psi_L: \langle F_1,\dots,F_N\rangle_L \to \mathbb{F}^{N} \right)$ defined by 
    $$F\mapsto (F({\bf{a_1}}),\dots,\Delta_{{\bf{a_1}}}^{m_1}F({\bf{a_1}}),\dots,
    F({\bf{a_k}}),\dots,\Delta_{{\bf{a_k}}}^{m_k}F({\bf{a_k}}))$$
    $$\left(\ F\mapsto (F_L({\bf{a_1}}),\dots,(\Delta_{{\bf{a_1}},L}^{m_1}F)_L({\bf{a_1}}),\dots,
    F_L({\bf{a_k}}),\dots,(\Delta_{{\bf{a_k}},L}^{m_k}F)_L({\bf{a_k}}))\ \right)$$
is a left (right) $\mathbb{F}$-module isomorphism, where $\langle F_1,\dots,F_N\rangle\ (\langle F_1,\dots,F_N\rangle_L)$ is the left (right) $\mathbb{F}$-module generated by the $F_j's$.

\item[$3)$] given any set of $N$ values in $\mathbb{F}$
$$\{b_{j,o},b_{j,x_{j_1}},b_{j,x_{j_2}x_{j_1}},\dots,b_{j,m_j}:j=1,2,\dots,k\}\ ,$$ 
$$\left(\ \{b_{j,o},b_{j,x_{j_{s(j)}}},b_{j,x_{j_{s(j)}}x_{j_{s(j)-1}}},\dots,b_{j,m_j}:j=1,2,\dots,k\}\ \right)\ ,$$
there exists a skew polynomial $F\in \mathcal{R}$ with $\deg(F)< N$ such that 
$$F({\bf{a_j}})=b_{j,o}, \Delta_{{\bf{a_j}}}^{x_{j_1}}F({\bf{a_j}})=b_{j,x_{j_1}},\Delta_{{\bf{a_j}}}^{x_{j_2}x_{j_1}}F({\bf{a_j}})=b_{j,x_{j_2}x_{j_1}}, \dots , \Delta_{{\bf{a_j}}}^{m_j}F({\bf{a_j}})=b_{j,m_j}$$ 
$$\left(F_L({\bf{a_j}})=b_{j,o}, (\Delta_{{\bf{a_j}},L}^{x_{j_{s(j)}}}F)_L({\bf{a_j}})=b_{j,x_{j_{s(j)}}}, \dots , (\Delta_{{\bf{a_j}},L}^{m_j}F)_L({\bf{a_j}})=b_{j,m_j}\right)$$
for all $j=1,\dots ,k$.
\end{itemize}
\end{theorem}

\begin{proof}
For simplicity, we will give here only the proof of the left version of the statement. The proof of its right version is very similar to that given here, keeping in mind the usual and suitable changes of notation. 

\smallskip

First of all, note that the equivalence between $2)$ and $3)$ is immediate.

\smallskip

\noindent $1)\Rightarrow 2)$ From Lemma  \ref{Der1.5}, it is evident that $\psi$ is a left $\mathbb{F}$-module homomorphism. Let ${\bf a_1}:=(a_{1_1},a_{1_2}, \dots ,a_{1_n})\in \mathbb{F}^{n}$. We start by defining the skew polynomial $G_{1,0} := 1$. Then, we see that $G_{1,0}({\bf a_1})=1$, $\deg(G_{1,0})=0<1$ and $\Delta_{{\bf a_1}}^{x_{1_j}\cdots x_{1_2}x_{1_1}}G_{1,0}({\bf a_1})=0$ for all $j=1,\dots,\deg(m_1)$. On the other hand, note that the skew polynomials 
$$G_{1,j}:=(x_{1_j}-a_{1_{j}})\cdots (x_{1_2}-a_{1_2})(x_{1_1}-a_{1_1})\in \mathbb{F}[x;\sigma,\delta]$$ 
are such that $G_{1,j}\in I^{x_{1_{j-1}}\cdots x_{1_2}x_{1_1}}(\{{\bf a_1}\})\setminus I^{x_{1_{j}}\cdots x_{1_2}x_{1_1}}(\{{\bf a_1}\})$, $\deg(G_{1,j})<j+1$ and $\Delta_{{\bf a_1}}^{x_{1_j}\cdots x_{1_2}x_{1_1}}G_{1,j}({\bf a_1})=1$ for all $j=1,\dots,\deg(m_1)$, where $I^{x_{1_{0}}}(\{{\bf a_1}\}):=I(\{{\bf a_1}\})$. 

Let ${\bf a_2}:=(a_{2_1},a_{2_2}, \dots ,a_{2_n})\in \mathbb{F}^{n}$. Since $\Omega$ is left DP-independent of type $(m_1, \dots ,m_k)$, by Proposition \ref{prop 3.1.7} and Lemma \ref{3.1.9}  there exists $F_{2,0}\in I^{m_1}(\{{\bf a_1}\})\setminus I^{m_1,o}(\{{\bf a_1},{\bf a_2}\})$ such that $\deg(F_{2,0})\leq \deg(m_1)+1$
and we can construct polynomials $G_{2,i}\in \mathcal{R}$ for $i=1,\dots,\deg (m_2)$ such that 
$G_{2,i}\in I^{m_1,x_{2_{i-1}}\cdots x_{2_2}x_{2_1}}(\{{\bf a_1},{\bf a_2}\})\setminus I^{m_1,x_{2_{i}}\cdots x_{2_2}x_{2_1}}(\{{\bf a_1},{\bf a_2}\})$, $\Delta_{{\bf a_2}}^{x_{2_i}\cdots x_{2_2}x_{2_1}}G_{2,i}({\bf a_2})=1$ and $\deg (G_{2,i})\leq \deg (m_1) + 1 + i$, for all $i=2,\dots,\deg(m_2)$. Then, by arguing as above, by Proposition \ref{prop 3.1.7} and Lemma \ref{3.1.9} we can construct for all ${\bf a_j}\in \Omega$ with $j=1,\dots,k$, skew polynomials  $G_{j,0},G_{j,1},\dots,G_{j,\deg(m_j)}\in \mathcal{R}$ such that
$$
\begin{array}{rcl}
  \psi(G_{j,0}) & = & (0, \dots ,0, \dots ,1,*,*, \dots ,*,*, \dots ,*, \dots ,*)
  \\ \psi(G_{j,1}) & = & (0, \dots ,0, \dots ,0,1,*, \dots ,*,*, \dots ,*, \dots ,*)
  \\ \vdots & = & \vdots
  \\ \psi(G_{j,\deg(m_j)}) & = & (0, \dots ,0, \dots ,0,0,0, \dots ,1,*, \dots ,*, \dots ,*)
\end{array}
$$
Thus, making left linear operations on all the polynomials $G_{j,0},G_{j,1},\dots,G_{j,\deg(m_j)}$ for $j=1,\dots,k$, we can obtain polynomials $\tilde{G}_{j,0},\tilde{G}_{j,1},\dots,\tilde{G}_{j,\deg(m_j)}\in \mathcal{R}$ such that
$$
\begin{array}{rcccl}
  \psi(\tilde{G}_{j,0}) & = & \vec{e}_{j,0} &:= & (0, \dots ,0, \dots ,1,0,0, \dots ,0, \dots ,0, \dots ,0)
  \\ \psi(\tilde{G}_{j,1})& = & \vec{e}_{j,1} &:= & (0, \dots ,0, \dots ,0,1,0, \dots ,0, \dots ,0, \dots ,0)
  \\ \vdots & = & \vdots & := & \vdots
  \\ \psi(\tilde{G}_{j,\deg(m_j)})& = & \vec{e}_{j,\deg(m_j)} &:= & (0, \dots ,0, \dots ,0,0,0, \dots ,1, \dots ,0, \dots ,0)
\end{array}
$$
for all $j=1, \dots ,k$. Therefore, given any $${\bf{b}}=(b_{1,0},b_{1,1}, \dots ,b_{1,\deg(m_1)}, \dots ,b_{j,0},b_{j,1}, \dots ,b_{j,\deg(m_j)}, \dots ,b_{k,0},b_{k,1}, \dots ,b_{k,\deg(m_k)})\in \mathbb{F}^{N}$$ it follows that
\begin{align*}
         {\bf{b}} &= \sum_{i=0}^{\deg(m_1)}b_{1,i}\vec{e}_{1,i}+ \dots +\sum_{i=0}^{\deg(m_j)}b_{j,i}\vec{e}_{j,i}+ \dots +\sum_{i=0}^{\deg(m_k)}b_{k,i}\vec{e}_{k,i} \\
                 &= \sum_{i=0}^{\deg(m_1)}b_{1,i}\psi(\tilde{G}_{1,i})+ \dots +\sum_{i=0}^{\deg(m_j)}b_{j,i}\psi(\tilde{G}_{j,i})+ \dots +\sum_{i=0}^{\deg(m_k)}b_{k,i}\psi(\tilde{G}_{k,i}) \\
                 &= \psi\left(\sum_{t=1}^{k}\left(\sum_{i=0}^{\deg(m_t)}b_{t,i}\tilde{G}_{t,i}\right)\right)
         \end{align*}
with $\deg \left(\sum_{t=1}^{k}\left(\sum_{i=0}^{\deg(m_t)}b_{t,i}\tilde{G}_{t,i}\right)\right)\leq N-1<N$ and we are done.

\smallskip

\noindent $2)\Rightarrow 1)$ Since $\psi$ is a surjective left $\mathbb{F}$-module homomorphism, then for each $\vec{e}_{j,0}\in \mathbb{F}^{N}$ with $j=1, \dots ,k$, as before, there exists a skew polynomial 
$$F_{j,0}\in 
I^{\vec{m}_j}\left(\Omega_{(j)}\right) \supsetneq I^{\vec{m}_j,o}(\Omega_{(j)}\cup\{{\bf a_j}\})\supsetneq \dots \supsetneq I^{\vec{m}_j,m_j}(\Omega_{(j)}\cup\{{\bf a_j}\})$$
for all $j=1, \dots ,k$, where $\vec{m}=(m_1,\dots,m_k)$. Hence, from Definition \ref{pd-independent} (2) it follows that $\Omega$ is left DP-independent of type $(m_1, \dots ,m_k)$. 
\end{proof}

In the special situation when $\Omega$ is left (right) DP-independent of type $(o, \dots ,o)$, i.e. $\Omega$ is left (right) P-independent (see Definition \ref{pd-independent}(3)), we obtain a necessary and sufficient condition to solve a skew Lagrange-type interpolation problem as follows.

\begin{corollary}[\textbf{Skew Lagrange-type interpolation}]\label{Tlagrange}
 Let $\Omega=\{{\bf{a_1}},\dots,{\bf{a_k}}\} \subseteq \mathbb{F}^n$ be a finite set. Then, the following conditions are equivalent:
 \begin{itemize}
     \item[$1)$] $\Omega$ is left (right) P-independent;
     \item[ $2)$] there exist $F_i\in\mathcal{R}$ with $\deg(F_i)<k$ for $i=1,\dots,k$, such that the map $\psi: \langle F_1,\dots,F_k\rangle 
    \to \mathbb{F}^{k}$ $\left(\psi_L: \langle F_1,\dots,F_k\rangle_L \to \mathbb{F}^{k} \right)$ defined by 
    $$F\mapsto (F({\bf{a_1}}), F({\bf{a_2}}),\dots, F({\bf{a_k}}))$$
    $$\left(\ F\mapsto (F_L({\bf{a_1}}), F_L({\bf{a_2}}),\dots, F_L({\bf{a_k}}))\ \right)$$
    is a left (right) $\mathbb{F}$-module isomorphism, where $\langle F_1,\dots,F_k\rangle\ (\langle F_1,\dots,F_k\rangle_L)$ is the left (right) $\mathbb{F}$-module generated by the $F_j's$ ;
\item[$3)$] given any set $\{b_{1},b_{2},\dots,b_{k}\}$ of $k$ values in $\mathbb{F}$, there exists a skew polynomial $F\in \mathcal{R}$ with $\deg(F)< k$ such that 
$$F({\bf{a_1}})=b_{1}, F({\bf{a_2}})=b_{2}, \dots , F({\bf{a_k}})=b_{k}$$ 
$$\left(\ F_L({\bf{a_1}})=b_{1}, F_L({\bf{a_2}})=b_{2}, \dots , F_L({\bf{a_k}})=b_{k}\ \right)\ .$$
\end{itemize}  
\end{corollary}

Let us give here also a result which allows us to construct recursively left 
(right) DP-independent sets. 

\begin{proposition}\label{3.1.11}
Let $\Omega=\{{\bf{a_1}},\dots,{\bf{a_k}}\} \subseteq \mathbb{F}^n$ be a left (right) DP-independent set of type $\vec{m}:=(m_1,\dots,m_{k})$. Consider $m:=x_{i_t}\cdots x_{i_2}x_{i_1}\in\mathcal{M}$ and define the following sets
$$E_j:=\{{\bf x}\in\mathbb{F}^n\ |\  \exists \lambda\in\mathbb{F}^* : ({\bf x}-{\bf x}^\lambda)\lambda=\vec{e}_j\}\ ,\ E_j^L:=\{{\bf x}\in\mathbb{F}^n\ |\  \exists \mu\in\mathbb{F}^* : \mu({\bf x}-{}^{\mu}{\bf x})=\vec{e}_j\}\ ,$$
where $\vec{e}_j$ is the $j$-th canonical vector of $\mathbb{F}^n$.
If 
$${\bf a_{k+1}}\notin Z(I^{\vec{m}}(\Omega))\cup \left(\bigcup_{j=1}^{t} E_{i_j}\right)\qquad \left(\ {\bf a_{k+1}}\notin Z_L(I_L^{\vec{m}}(\Omega))\cup \left(\bigcup_{j=1}^{t} E_{i_j}^L\right)\ \right)\ ,$$ 
then $\Omega\cup\{ {\bf a_{k+1}}\}$ is a left (right) DP-independent set of type $(\vec{m},m)$.
\end{proposition}

\begin{proof} Since ${\bf a_{k+1}}\notin Z(I^{\vec{m}}(\Omega))\ \left(Z_L(I_L^{\vec{m}}(\Omega)) \right)$, we deduce that there exists $F\in I^{\vec{m}}(\Omega) \setminus I^{\vec{m},o}(\Omega\cup\{{\bf a_{k+1}}\})$ $\left(\ I^{\vec{m}}_L(\Omega) \setminus I^{\vec{m},o}_L(\Omega\cup\{{\bf a_{k+1}}\})\ \right)$, implying that
$$I^{\vec{m}}(\Omega) \supsetneq I^{\vec{m},o}(\Omega\cup\{{\bf a_{k+1}}\})\ \left(\ I^{\vec{m}}_L(\Omega) \supsetneq I^{\vec{m},o}_L(\Omega\cup\{{\bf a_{k+1}}\})\ \right)\ .$$ 
Without loss of generality, we can assume that $F({a_{k+1}})=1\ \left(F_L({a_{k+1}})=1 \right)$. Consider the skew polynomial $H_s:=(x_s-({\bf a_{k+1}})_s)F\ \left(H_s:=F(x_s-({\bf a_{k+1}})_s)\right)$. Note that $H_s\in I^{\vec{m},o}\left(\Omega\cup \{{\bf a_{k+1}} \}\right)\ \left(I_L^{\vec{m},o}\left(\Omega\cup \{{\bf a_{k+1}} \}\right) \right)$, for any $s=1,\dots,n$. Moreover, we have 
$$\Delta_{{a_{k+1}}}^{x_{i_1}}H_s=(x_s-({\bf a_{k+1}})_s)\Delta_{{a_{k+1}}}^{x_{i_1}}F+\Delta_{{a_{k+1}}}^{x_{i_1}}\big( \left(x_s-({\bf a_{k+1}})_s\right)\big)$$ 
$$\left(\ \Delta_{{a_{k+1}},L}^{x_{i_1}}H_s=\left(\Delta_{{a_{k+1}},L}^{x_{i_1}}F\right)(x_s-({\bf a_{k+1}})_s)+\Delta_{{a_{k+1}},L}^{x_{i_1}}\big( \left(x_s-({\bf a_{k+1}})_s\right)\big)\ \right)\ .$$
Writing $\mu:=\Delta_{{a_{k+1}}}^{x_{i_1}}F({\bf a_{k+1}})\ \left(\mu_L:=\left(\Delta_{{a_{k+1}},L}^{x_{i_1}}F\right)_L({\bf a_{k+1}}) \right)$, we see that 
$$\Delta_{{a_{k+1}}}^{x_{i_1}}H_s({\bf a_{k+1}})=(({\bf a_{k+1}}^{\mu})_s-({\bf a_{k+1}})_s)\mu+\Delta_{{a_{k+1}}}^{x_{i_1}}\big( \left(x_s-({\bf a_{k+1}})_s\right)\big)\neq 0$$ 
$$\left(\ \left(\Delta_{{a_{k+1}}}^{x_{i_1}}H_s\right)_L({\bf a_{k+1}})=\mu_L(({}^{\mu_L}{\bf a_{k+1}})_s-({\bf a_{k+1}})_s)+\Delta_{{a_{k+1}},L}^{x_{i_1}}\big( \left(x_s-({\bf a_{k+1}})_s\right)\big)\neq 0\ \right)\ ,$$ 
for some index $s\in\{1,\dots,n\}$, 
because 
$$({\bf a_{k+1}}^{\xi}-{\bf a_{k+1}})\xi+\vec{e}_{i_1}\neq \vec{0}\ \ \ \ \left(\ \xi({}^{\xi}{\bf a_{k+1}}-{\bf a_{k+1}})+\vec{e}_{i_1}\neq \vec{0}\ \right)$$ for every $\xi\in\mathbb{F}^*$. Thus, normalizing $H_s$, we see that there exists 
$$H\in I^{\vec{m},o}\left(\Omega\cup \{{\bf a_{k+1}} \}\right)\quad \left(I_L^{\vec{m},o}\left(\Omega\cup \{{\bf a_{k+1}} \}\right) \right)$$
such that $\Delta_{{a_{k+1}}}^{x_{i_1}}H({\bf a_{k+1}})=1\ \ \left(\ \left(\Delta_{{a_{k+1}}}^{x_{i_1}}H_s\right)_L({\bf a_{k+1}})=1\ \right)$, i.e.
$$I^{\vec{m},o}(\Omega\cup\{{\bf a_{k+1}}\}) \supsetneq I^{\vec{m},x_{i_1}}(\Omega\cup\{{\bf a_{k+1}}\})\ \left(\ I_L^{\vec{m},o}(\Omega\cup\{{\bf a_{k+1}}\}) \supsetneq I_L^{\vec{m},x_{i_1}}(\Omega\cup\{{\bf a_{k+1}}\})\ \right)\ .$$
By an inductive argument, we can obtain that 
$$I^{\vec{m}}(\Omega)\supsetneq I^{\vec{m},o}(\Omega\cup\{{\bf a_{k+1}}\}) \supsetneq I^{\vec{m},x_{i_1}}(\Omega\cup\{{\bf a_{k+1}}\})\supsetneq \dots \supsetneq I^{\vec{m},m}(\Omega\cup\{{\bf a_{k+1}}\})$$
$$\left(\ I_L^{\vec{m}}(\Omega)\supsetneq I_L^{\vec{m},o}(\Omega\cup\{{\bf a_{k+1}}\}) \supsetneq I_L^{\vec{m},x_{i_1}}(\Omega\cup\{{\bf a_{k+1}}\})\supsetneq \dots \supsetneq I_L^{\vec{m},m}(\Omega\cup\{{\bf a_{k+1}}\})\ \right)\ .$$
This shows that ${\bf a_{k+1}}$ is left (right) DP-independent of type $(\vec{m},m)$ from $\Omega$. Since $\Omega$ is left (right) DP-independent of type $\vec{m}$, we get case 2) of Theorem \ref{theorem 3.1.7}, which is equivalent to say that $\Omega\cup\{ {\bf a_{k+1}}\}$ is a left (right) DP-independent set of type $(\vec{m},m)$. 
\end{proof}

\begin{remark}
When $\Omega\subseteq\mathbb{F}^n$ is left (right) DP-independent of type $(o, \dots ,o)$, i.e. $\Omega$ is left (right) P-independent, and ${\bf{a}}\in \mathbb{F}^{n}\setminus \Omega$ is such that  $I(\Omega) \supsetneq I (\Omega\cup\{ {\bf a}\})$ $\left(\ I_L(\Omega) \supsetneq I_L (\Omega\cup\{ {\bf a}\})\ \right)$, which is equivalent to require that ${\bf a}\notin Z(I(\Omega))\ \left({\bf a}\notin Z(I_L(\Omega))\right)$,
Proposition \ref{3.1.11} shows that $\Omega\cup \{{\bf a}\}$ is left (right) P-independent, giving \cite[Lemma 36]{evaluationandinterpolation} when $\Omega$ is a (finite) P-independent set.
\end{remark}

\smallskip

Finally, let us show here that the ring $\mathcal{R}$ is not noetherian for $n\geq 2$, unlike for $n=1$.

\smallskip

\begin{example}
Consider $\mathcal{R}:=\mathbb{F}[x_1,x_2;\sigma,\delta]$ and the left (right) ideals $$I_k:=\mathcal{R}x_1x_2+\mathcal{R}x_1x_2^2+ \dots +\mathcal{R}x_1x_2^{k}\;\;(I_{k,R}:=x_1x_2\mathcal{R}+x_1x_2^2\mathcal{R}+ \dots +x_1x_2^{k}\mathcal{R})$$ with $k\in\mathbb{Z}_{\geq 1}$. We note that 
$$I_1\subsetneq I_2 \subsetneq \dots \subsetneq I_k \subsetneq \mathcal{R}\ \ \ (I_{1,R}\subsetneq I_{2,R} \subsetneq \dots \subsetneq I_{k,R}\subsetneq \mathcal{R})$$ 
because the variables $x_1,x_2$ do not commute and $x_1\notin I_k$ ($I_{k,R}$) for every $k\in\mathbb{Z}_{\geq 1}$. Then $\mathcal{R}$ does not satisfy the ascending chain condition on left (right) ideals and therefore it is not a noetherian ring. In particular, when $\sigma=Id$ and $\delta=0$, the free conventional polynomial ring $\mathbb{F}[\textbf{x}; Id, 0]$ is not noetherian.
\end{example}

\subsection{The univariate case $(n=1)$}
In what follows, we will specialize the main results of the previous subsection
to the case $n=1$. For simplicity and to avoid any confusion of notation, we will use here the symbol $\mathcal{A}$ instead of $\mathcal{R}$ for denoting the noetherian ring of univariate skew polynomials, and we define 
$$\Delta_{a}^{0}F(x):=\Delta_{a}^{o}F(x)=F(x) \qquad \left(\ \Delta_{a,L}^{0}F(x):=\Delta_{a,L}^{o}F(x)=F(x)\ \right)$$
$$\Delta_{a}^{t}F(x):=\Delta_{a}^{x^t}F(x) \qquad \left(\ \Delta_{a,L}^{t}F(x):=\Delta_{a,L}^{x^t}F(x)\ \right) $$
and
$$\Delta_{a}^{0}F(a):=\Delta_{a}^{o}F(a)=F(a) \quad \left(\ \left(\Delta_{a,L}^{0}F\right)_L(a):=\left(\Delta_{a,L}^{o}F\right)_L(a)=F_L(a)\ \right)$$
$$\Delta_{a}^{t}F(a):=\Delta_{a}^{x^t}F(a) \quad \left(\ \left(\Delta_{a,L}^{t}F\right)_L(a):=\left(\Delta_{a,L}^{x^t}F\right)_L(a)\ \right) $$
for any $m=x^t\in\mathcal{M}$ with $t\in\mathbb{Z}_{>0}$ and every $ a\in\mathbb{F}$. Moreover, if $S\subseteq \mathbb{F}$, then let $f_S^{\vec{m}}\ (f_{S,L}^{\vec{m}})$ be the generator of $I^{\vec{m}}(S)\ \left(I^{\vec{m}}_L(S) \right)$, i.e. $I^{\vec{m}}(S)=\mathcal{A}f_S^{\vec{m}}\ \left(I^{\vec{m}}_L(S)=f_{S,L}^{\vec{m}}\mathcal{A}\right)$. In particular, when $\vec{m}=\vec{0}$, we will simply write $f_S:=f_S^{\vec{0}}\ \left(f_{S,L}:=f_{S,L}^{\vec{0}} \right)$.

\begin{theorem}[\textbf{Skew univariate Hermite-type interpolation}]\label{HERMITE N=1} 
Let $\sigma$ be an endomorphism (automorphism) of $\mathbb{F}$, $\Omega=\{{{a_1}},\dots,{{a_k}}\} \subseteq \mathbb{F}$ a finite set,  $\vec{m}=(m_1,\dots,m_k)\in \mathbb{Z}^{k}_{\geq 0}$ and define $N:=\sum_{j=1}^k(m_j+1)$. The following conditions are equivalent:
\begin{itemize}
    \item[$1)$] $\deg(f_{\Omega}^{\vec{m}})=N$ $(\deg(f_{\Omega,L}^{\vec{m}})=N)$ ;
    \item[$2)$] the map $\psi: \mathcal{A}_{N} \to \mathbb{F}^{N}$ $\left(\psi_L: \mathcal{A}_{N} \to \mathbb{F}^{N} \right)$ defined by $$f\mapsto (f({{a_1}}),\dots,\Delta_{{{a_1}}}^{m_1}f({{a_1}}),\dots,
    f({{a_k}}),\dots,\Delta_{{{a_k}}}^{m_k}f({{a_k}}))$$
    $$\left(f\mapsto (f_L({{a_1}}),\dots,(\Delta_{{{a_1},L}}^{m_1}f)_L({{a_1}}),\dots,
     f_L({{a_k}}),\dots,(\Delta_{{{a_k},L}}^{m_k}f)_L({{a_k}}))\right)$$
is a left (right) $\mathbb{F}$-module isomorphism, where  $\mathcal{A}_{N}:=\{f\in \mathcal{A}:\deg f<N\}$ ;
 \item[$3)$]  given any finite set of $N$ values in $\mathbb{F}$
$$\{b_{j,0},b_{j,1},b_{j,2},\dots,b_{j,m_j}:j=1,2,\dots,k\}\ ,$$ there exists a unique skew polynomial $f\in \mathcal{A}$ with $\deg(f)< N$ such that 
$$\qquad f({{a_j}})=b_{j,0}, \Delta_{{{a_j}}}^{1}f({{a_j}})=b_{j,1},\Delta_{{{a_j}}}^{2}f({{a_j}})=b_{j,2}, \dots , \Delta_{{{a_j}}}^{m_j}f({{a_j}})=b_{j,m_j}$$
$$\left(f_L({{a_j}})=b_{j,0}, (\Delta_{{{a_j},L}}^{1}f)_L({{a_j}})=b_{j,1},(\Delta_{{{a_j},L}}^{2}f)_L({{a_j}})=b_{j,2}, \dots , (\Delta_{{{a_j},L}}^{m_j}f)_L({{a_j}})=b_{j,m_j}\right)$$
for all $j=1,\dots,k$ ;
\item[$4)$] $\Omega$ is left (right) DP-independent  of type $(m_1, \dots ,m_k)$ .
\end{itemize}

\medskip

\noindent In particular, $\Omega$ is left (right) DP-independent of type ${\vec{m}}:=(m_1, \dots ,m_k)$ if and only if 
$$\deg(f_{\Omega}^{\vec{m}})=N \ \ \left(\deg(f_{\Omega,L}^{\vec{m}})=N\right)\ .$$
\end{theorem}

\smallskip

\begin{proof}
We will only show the equivalence between the statements for the left case, because for the right case the proofs are quite similar.

\smallskip

\noindent $1) \Rightarrow 2)$  From Lemma \ref{Der1.5} for $n=1$, it is evident that $\psi$ is a left $\mathbb{F}$-module homomorphism. We show that $\psi$ is injective. Let $f \in \mathcal{A}_N$ with $f\neq 0$ and assume that $\psi(f) = 0$. Then $f\in I^{\vec{m}}(\Omega)=\mathcal{A}f_{\Omega}^{\vec{m}}$ and therefore, $f=q\cdot f_{\Omega}^{\vec{m}}$ for some non-zero skew polynomial $q\in \mathcal{A}$. Thus, $\deg f\geq \deg f_{\Omega}^{\vec{m}}=N$,  which is absurd because $f \in \mathcal{A}_{N}$. Therefore $f = 0$ and $\psi $ is injective. On the other hand, since $\mathcal{A}/\ker \psi= \mathcal{A}/\mathcal{A}f_{\Omega}^{\vec{m}}\cong Im(\psi)$, it follows that $\dim Im(\psi)=\dim \mathcal{A}/\mathcal{A}f_{\Omega}^{\vec{m}}=\deg f_{\Omega}^{\vec{m}}=N=\dim \mathbb{F}^{N} $ and therefore $\psi$ is a surjective left $\mathbb{F}$-module homomorphism.

\smallskip

\noindent $2)\Rightarrow 1)$ Since $\psi $ is a surjective left $\mathbb{F}$-module homomorphism and $\mathcal{A}/\mathcal{A}f_{\Omega}^{\vec{m}}=\mathcal{A}/\ker \psi\cong Im(\psi)$, it follows that  $\deg (f_{\Omega}^{\vec{m}})=\dim Im(\psi)=\dim \mathbb{F}^{N}=N$.

\smallskip

\noindent $2)\Leftrightarrow 3)$ Immediate.

\smallskip

\noindent Finally, the equivalence between $3)$ and $4)$ follows from Theorem \ref{theorem 3.1.7} for $n=1$ and the fact that if there exists another polynomial $g\in \mathcal{A}$ with $\deg g < N$ such that
$$\qquad g({{a_j}})=b_{j,0}, \Delta_{{{a_j}}}^{1}g({{a_j}})=b_{j,1},\Delta_{{{a_j}}}^{2}g({{a_j}})=b_{j,2}, \dots , \Delta_{{{a_j}}}^{m_j}g({{a_j}})=b_{j,m_j}$$
for all $j=1,\dots,k$, then
$$(f-g) \in \ker \psi= I^{\vec{m}}(\Omega) =\mathcal{A}f_{\Omega}^{\vec{m}}.$$ Note that  $\mathcal{A}/\ker \psi \cong Im(\psi)=\mathbb{F}^N$, because $\psi$ is surjective by Theorem \ref{theorem 3.1.7}. Hence, $\deg (f_{\Omega}^{\vec{m}})=N$ and therefore $f=g+qf_{\Omega}^{\vec{m}}$ for some $q\in \mathcal{A}$. Since $\deg (f)<  N$ and $\deg (g)<N$, this implies $q=0$, i.e. $f=g$.
\medskip
\end{proof}

\begin{corollary}[\textbf{Skew univariate Lagrange interpolation}]\label{LAGRANGE N=1}
 Let $\sigma$ be an endomorphism (automorphism) of $\mathbb{F}$ and let $\Omega=\{{{a_1}},a_2,\dots,{{a_k}}\} \subseteq \mathbb{F}$ be a finite set. The following conditions are equivalent:
 \begin{itemize}
 \item[$1)$] $\deg(f_{\Omega})=k$ $(\deg(f_{\Omega,L})=k)$ ;
     \item[ $2)$] the map $\phi:\mathcal{A}_k \to \mathbb{F}^{k}$ $(\phi_L:\mathcal{A}_k \to \mathbb{F}^{k})$, defined by $$f \mapsto (f({{a_1}}),f({{a_2}}), \dots ,f({{a_k}}))$$
     $$\left(f \mapsto (f_L({{a_1}}),f_L({{a_2}}), \dots ,f_L({{a_k}}))\right)$$
     is a left (right) $\mathbb{F}$-module isomorphism ;
     \item[$3)$] for every $b_1, \dots , b_k\in \mathbb{F}$, there exists a unique skew polynomial $f\in \mathcal{A}$ with $\deg(f)<k$ such that $f({{a_j}})=b_{j}$ $(f_L({{a_j}})=b_{j})$ for all $j=1, \dots ,k$ ;
     \item[$4)$] $\Omega$ is left (right) P-independent .
     
     \medskip
     
 \end{itemize}
 
 \noindent In particular, $\Omega$ is left (right) P-independent if and only if $\deg(f_{\Omega})=k$ $(\deg(f_{\Omega,L})=k)$.  
\end{corollary}

\smallskip

\begin{remark}
For the left case and $\delta=0$, the equivalences between 1) and 2) in Theorem \ref{HERMITE N=1} and Corollary \ref{LAGRANGE N=1} follows from \cite[Theorem 4.4]{Er1} and \cite[Theorem 4.2]{Er1}, respectively.
\end{remark}

\begin{example}
Let $\mathbb{C}[x;\sigma,0]$ with $\sigma(z)=\bar{z}$ (the complex conjugation) for all $z\in \mathbb{C}$ and let $\Omega=\{i,1+i\}\subseteq \mathbb{C}$. Since $[i]_L=[i]\neq [1+i]=[1+i]_L$, we see that $\Omega$ is left (right) DP-independent of type $(m_1,m_2)$ for any $m_1,m_2\in \mathbb{Z}_{\geq 0}$. We will construct skew polynomials $f(x), g(x)\in \mathbb{C}[x;\sigma,0]$ of degree $< 5$ such that $$f(i)=f(1+i)=0,\;\Delta_{i}^{1}f(i)=\Delta_{1+i}^{1}f(1+i)=1\;\text{and}\;\Delta_{i}^{2}f(i)=i\ ,$$
$$g_L(i)=g_L(1+i)=0,\;\left(\Delta_{i,L}^{1}\ g\right)_L(i)=\left(\Delta_{1+i,L}^{1}\ g\right)_L(1+i)=1\;\text{and}\;\left(\Delta_{i,L}^{2}\ g\right)_L(i)=i\ .$$
By Theorem \ref{HERMITE N=1}, the maps $\psi, \psi_L: \mathbb{C}[x;\sigma,0]_{5} \to \mathbb{F}^{5}$ defined by 
$$f\mapsto \psi(f):=\left(f(i),\Delta_{i}^{1}f(i),\Delta_{i}^{2}f(i),f(1+i),\Delta_{1+i}^{1}f(1+i)\right)\ ,$$
$$g\mapsto \psi_L(g):=\left(g_L(i),\left(\Delta_{i,L}^{1}g\right)_L(i),\left(\Delta_{i,L}^{2}g\right)_L(i),g_L(1+i),\left(\Delta_{1+i,L}^{1}g\right)_L(1+i)\right)\ ,$$
are left and right $\mathbb{F}$-module isomorphisms respectively, where 
$$\mathbb{C}[x;\sigma,0]_{5}:=\{h\in \mathbb{C}[x;\sigma,0] \ | \ \deg h<5\}\ .$$ 
Note that $$
\begin{array}{ccccc}
  \psi(F_{1,0}(x)) & = & (1,0,0,1,0) & = & \psi_L(G_{1,0}(x)) \\ 
  \psi(F_{1,1}(x)) & = & (0,1,0,1,1) & = & \psi_L(G_{1,1}(x)) \\
  \psi(F_{1,2}(x)) & = & (0,0,1,1,2) & = & \psi_L(G_{1,2}(x)) \\ 
 \psi(F_{2,0}(x)) & = & (0,0,0,1,3) & = & \psi_L(G_{2,0}(x)) \\
 \psi(F_{2,1}(x)) & = & (0,0,0,0,1) & = & \psi_L(G_{2,1}(x)) 
\end{array}\;$$
where 
$$F_{1,0}(x)=G_{1,0}(x):=1\ ,\ F_{1,1}(x)=G_{1,1}(x):=(x-i)\ ,\ F_{1,2}(x)=G_{1,2}(x):=(x-i)^2\ ,$$
$$F_{2,0}(x)=G_{2,0}(x):=(x-i)^3\ ,\ F_{2,1}(x):=(x-1-i)(x-i)^3\ ,$$
$$G_{2,1}(x):=(x-i)^3(x-1-i)\ .$$
Then, making left (right) linear operations on the previous skew polynomials, we have 
$$
\begin{array}{ccccc}
  \psi(F'_{1,0}(x)) & = & \psi_L(G'_{1,0}(x)) & = & (1,0,0,0,0) \\ 
  \psi(F'_{1,0}(x)) & = & \psi_L(G'_{1,1}(x)) & = & (0,1,0,0,0) \\ 
  \psi(F'_{1,0}(x)) & = & \psi_L(G'_{1,2}(x)) & = & (0,0,1,0,0) \\ 
  \psi(F'_{1,0}(x)) & = & \psi_L(G'_{2,0}(x)) & = & (0,0,0,1,0) \\
  \psi(F'_{1,0}(x)) & = & \psi_L(G'_{2,1}(x)) & = & (0,0,0,0,1)
\end{array}
$$
where 
$$F'_{1,0}(x):=F_{1,0}(x)-F_{2,0}(x)+3\cdot F_{2,1}(x)\ ,\ G'_{1,0}(x):=G_{1,0}(x)-G_{2,0}(x)+G_{2,1}(x)\cdot 3\ , $$
$$F'_{1,1}(x):=F_{1,1}(x)-F_{2,0}(x)+2\cdot F_{2,1}(x)\ ,\ G'_{1,1}(x):=G_{1,1}(x)-G_{2,0}(x)+G_{2,1}(x)\cdot 2\ ,$$
$$F'_{1,2}(x):=F_{1,2}(x)-F_{2,0}(x)+F_{2,1}(x)\ ,\ G'_{1,2}(x):=G_{1,2}(x)-G_{2,0}(x)+G_{2,1}(x)\ ,$$
$$F'_{2,0}(x):=F_{2,0}(x)-3\cdot F_{2,1}(x)\ ,\ G'_{2,0}(x):=G_{2,0}(x)-G_{2,1}(x)\cdot 3\ ,$$ 
$$F'_{2,1}(x):=F_{2,1}(x)\ ,\ G'_{2,1}(x):=G_{2,1}(x)\ .$$
Hence, we obtain that
$$
\begin{array}{rl}
   f(x) &= 0\cdot F'_{1,0}(x)+1\cdot F'_{1,1}(x)+i\cdot F'_{1,2}(x)+0\cdot F'_{2,0}(x)+1\cdot F'_{2,1}(x) \\
   &= (3+i)x^4-(4+2i)x^3-(8-3i)x^2+(5+2i)x+5-5i\ , \\[10pt]
   g(x) &= G'_{1,0}(x)\cdot 0+G'_{1,1}(x)\cdot 1+G'_{1,2}(x)\cdot i+G'_{2,0}(x)\cdot 0+G'_{2,1}(x)\cdot 1 \\
   &= x^4(3+i)+x^3(-4-2i)+x^2(-8+3i)+x(5+2i)+5-5i \\
   &= (3+i)x^4-(4-2i)x^3-(8-3i)x^2+(5-2i)x+5-5i\ .
\end{array}
$$
Finally, let us observe that 
$$
E_j:=\{{x}\in\mathbb{C}\ |\  \exists \lambda\in\mathbb{C}^* : ({x}-{x}^\lambda)\lambda=1\}=\{{x}\in\mathbb{C}\ |\  \exists \lambda\in\mathbb{C}^* : {x}\lambda-\overline{\lambda}{x}=1\}=$$
$$=\{{x}\in\mathbb{C}\ |\  \exists \lambda\in\mathbb{C}^* : (\lambda-\overline{\lambda}){x}=1\}=\{ bi \ |\ b\in\mathbb{R}^*\}=E_j^L\ .$$\end{example}

\medskip

The following example shows that there exist left (right) DP-independent sets of some type, whose elements belong to the same right (left) $(\sigma,\delta)$-conjugacy class.

\begin{example}
Let $\mathbb{C}[x;\sigma,0]$ with $\sigma(z)=\bar{z}$ for all $z\in \mathbb{C}$ and let $\Omega=\{a_1,a_2\}\subseteq \mathbb{C}$, where $a_1=1$ and $a_2$ is a complex number different from $\{1,\pm i\}$ which belongs to the unit circle of radius $1$, that is, $[a_1]=[a_2]$ $([a_1]_L=[a_2]_L)$, because $\delta=0$, $\mathbb{C}$ is commutative and $\sigma^{-1}=\sigma$. Let us show that $\Omega$ is left (right) DP-independent of type $(0,1):=(o,x)$. We need to prove that 
$$I^{0}(\{a_1\})\supsetneq I^{0,0}(\{a_1\}\cup \{a_2\})\;\;\text{and}\;\;I^{1}(\{a_2\})\supsetneq I^{1,0}(\{a_2\}\cup \{a_1\})$$
$$\left(I^{0}_L(\{a_1\})\supsetneq I^{0,0}_L(\{a_1\}\cup \{a_2\})\;\;\text{and}\;\;I^{1}_L(\{a_2\})\supsetneq I^{1,0}_L(\{a_2\}\cup \{a_1\})\right).$$
Note that $I^{0}(\{a_1\})\supsetneq I^{0,0}(\{a_1\}\cup \{a_2\})\ \ \left(I^{0}_L(\{a_1\})\supsetneq I^{0,0}_L(\{a_1\}\cup \{a_2\})\right)$ holds. On the other hand, observe that  $f(x):=(x-a_2)^2\in I^{1}(\{a_2\})\ \left(f(x):=(x-a_2)^2\in I^{1}_L(\{a_2\})\right)$ is such that $f(a_1)\neq 0\ \left(f_L(a_1)\neq 0\right)$. Indeed, if $f(a_1)=f(1)=(1^{1-a_2}-a_2)(1-a_2)=0\ \left(f_L(a_1)=(1-a_2)({}^{1-a_2}1-a_2)=0\right)$, then for both cases (right and left) we have
\begin{center}
    either $a_2=1$, or  $\overline{(1-a_2)}(1-a_2)^{-1}-a_2=0$,
\end{center}
because $\delta=0$, $\mathbb{C}$ is commutative and $\sigma^{-1}=\sigma$. 
The former case is not possible because $a_2\neq 1$. In the latest case, by writing $a_2:=\alpha+\beta i$ with $\alpha,\beta \in \mathbb{R}$, it follows that $$(\alpha^2-\beta^2-2\alpha+1)+2\alpha\beta i=0.$$
Thus, $2\alpha \beta=0 \Leftrightarrow \alpha=0$, or $\beta=0$. If $\alpha=0$ then we have $\beta=\pm 1$. Hence $a_2=\pm i$, but this is not possible. Finally, if $\beta=0$ then $a_2=1$, but this is impossible. Thus, we get that $f(a_1)\neq 0$ ($f_L(a_1)\neq 0$).
\end{example}

Note that if a finite set $\Omega\subseteq\mathbb{F}^n$ is left (right) DP-independent of some type, then $\Omega$ is also left (right) P-independent. To conclude with the case $n=1$, let us show that if a finite set $\Omega\subseteq \mathbb{F}$ is left (right) P-independent, then it is not necessarily left (right) DP-independent of some type. 

\begin{example}
Let $\mathbb{C}[x;\sigma,0]$ with $\sigma(z)=\bar{z}$ (the complex conjugation) for all $z\in \mathbb{C}$ and let $\Omega=\{1,-i\}\subseteq \mathbb{C}$. Clearly $\Omega$ is left (right) P-independent. However, there does not exist $(m_1,m_2)\in \mathbb{Z}_{\geq 0}\times \mathbb{Z}_{>0}\setminus \{(0,0)\}$ such that $\Omega$ is left (right) DP-independent of type $(m_1,m_2)$. Indeed, since every skew polynomial $f\in I^{m_2}(\{-i\})$ ($f\in I^{m_2}_L(\{-i\})$) is of the form $f(x)=g(x)(x+i)^{m_2+1}$ ($f(x)=(x+i)^{m_2+1}g(x)$) for some $g(x)\in \mathbb{C}[x;\sigma,0]$ and $(x+i)^2=x^2-1$, it follows that any power $\geq 2$ of $(x+i)$ is zero at $1$ and therefore there does not exist a skew polynomial $f\in I^{m_2}(\{-i\})\setminus I^{m_2,0}(\{-i\}\cup \{1\})$ ($f\in I^{m_2}_L(\{-i\})\setminus I^{m_2,0}_L(\{-i\}\cup \{1\})$). 
\end{example}

\subsection{Skew Vandermonde matrices} Let us define here skew confluent Vandermonde matrices for the multivariate case $n\geq 1$ (as to $n=1$, the univariate cases, see e.g. \cite{zerosmartinez}). 

\begin{definition}\label{DefVandermonde}
Fix $d\in \mathbb{Z}_{>0}$. Take $m_i:=x_{i_{s(i)}}\cdots x_{i_1}\in \mathcal{M}$ and consider the set 
$\{m'_j\}_{j=1,\dots,r}$ of all the monomials in $\mathcal{M}$ of degree less than $d$, with a certain ordering in $\mathcal{M}$. We define 
the \textit{right (left) $(\sigma,\delta)$-confluent Vandermonde matrix $V_{m_i}^{\sigma,\delta}({\bf{a_i}}) \ \left(V_{m_i,L}^{\sigma,\delta}({\bf{a_i}})\right)$ of order $d$ on ${\bf{a_i}}\in \mathbb{F}^{n}$} as the
$r\times \left(\deg \left(m_i\right)+1\right)\ \left( \left(\deg \left(m_i\right)+1\right)\times r \right)$ matrix with entries in $\mathbb{F}$ given by
$$V_{m_i}^{\sigma,\delta}({\bf{a_i}}):=\left(
\begin{matrix}
\Delta^{o}_{\bf{a_i}}m'_{1}(\bf{{a}}_{i})& \Delta^{x_{i_1}}_{\bf{a_i}}m'_{1}({\bf{a}_{i}})& \cdots & \Delta^{m_i}_{\bf{a_i}}m'_{1}({\bf{a_i}})\\
\Delta^{o}_{\bf{a_i}}m'_{2}(\bf{{a}}_{i}) & \Delta^{x_{i_1}}_{\bf{a_i}}m'_{2}(\bf{a_i})  & \cdots & \Delta^{m_i}_{\bf{a_i}}m'_{2}({\bf{a_i}}) \\
\vdots & \vdots & \ddots & \vdots \\
\Delta^{o}_{\bf{a_i}}m'_{r}(\bf{a_i}) & \Delta^{x_{i_1}}_{\bf{a_i}}m'_{r}(\bf{a_i}) & \cdots & \Delta^{m_i}_{\bf{a_i}}m'_{r}(\bf{a_i})
\end{matrix}
\right)$$
$$\left(\ V_{m_i,L}^{\sigma,\delta}({\bf{a_i}}):=\left(
\begin{matrix}
\left(\Delta^{o}_{\bf{a_i},L}m'_{1}\right)_L(\bf{{a}}_{i})& \left(\Delta^{o}_{\bf{a_i},L}m'_{2}\right)_L(\bf{{a}}_{i})& \cdots & \left(\Delta^{o}_{\bf{a_i},L}m'_{r}\right)_L(\bf{{a}}_{i})\\
\left(\Delta^{x_{i_1}}_{\bf{a_i},L}m'_{1}\right)_L(\bf{{a}}_{i})& \left(\Delta^{x_{i_1}}_{\bf{a_i},L}m'_{2}\right)_L(\bf{{a}}_{i})& \cdots & \left(\Delta^{x_{i_1}}_{\bf{a_i},L}m'_{r}\right)_L(\bf{{a}}_{i})\\
\vdots & \vdots & \ddots & \vdots \\
\left(\Delta^{m_i}_{\bf{a_i},L}m'_{1}\right)_L(\bf{{a}}_{i})& \left(\Delta^{m_i}_{\bf{a_i},L}m'_{2}\right)_L(\bf{{a}}_{i})& \cdots & \left(\Delta^{m_i}_{\bf{a_i},L}m'_{r}\right)_L(\bf{{a}}_{i})\\
\end{matrix}
\right)\ \right)\ .$$
Moreover, we define the \textit{right (left) $(\sigma,\delta)$-confluent Vandermonde matrix 
$$V^{\sigma,\delta}_{(m_1, \dots ,m_k)} ({\bf{a_1}}, {\bf{a_2}}, \dots , {\bf{a_k}}) \ \ \ \left(V^{\sigma,\delta}_{(m_1, \dots ,m_k)} ({\bf{a_1}}, {\bf{a_2}}, \dots , {\bf{a_k}})\right)$$
of order $d$ on the set $\{{\bf{a_1}},{\bf{a}_2}, \dots ,{\bf{a_k}}\}$} as the {\small{ $r\times \sum_{i=1}^{k}\left(\deg \left(m_i\right)+1\right)\ \left( \sum_{i=1}^{k}\left(\deg \left(m_i\right)+1\right)\times r\right)$} }matrix with entries in $\mathbb{F}$ given by appending column-wise (row-wise) the right (left) $(\sigma,\delta)$-confluent Vandermonde matrices of order $d$ on the ${\bf{a}_i}$'s, that is, 
$$V^{\sigma,\delta}_{(m_1, \dots ,m_k)} ({\bf{a_1}}, {\bf{a_2}}, \dots , {\bf{a_k}}) := (V_{m_1}^{\sigma,\delta}({\bf{a_1}})\ |\ V_{m_2}^{\sigma,\delta} ({\bf{a_2}})\ |\ \dots \ |\ V_{m_k}^{\sigma,\delta}({\bf{a_k}}))$$
$$\left(V^{\sigma,\delta}_{(m_1, \dots ,m_k),L} ({\bf{a_1}}, {\bf{a_2}}, \dots , {\bf{a_k}}) := \left(
\begin{matrix}
V_{m_1,L}^{\sigma,\delta}({\bf{a_1}})\\ 
V_{m_2,L}^{\sigma,\delta} ({\bf{a_2}})\\ 
\vdots \\ 
V_{m_k,L}^{\sigma,\delta}({\bf{a_k}})
\end{matrix}
\right)\ \right)\ .$$
\end{definition} 

\begin{remark}
The matrix $V^{\sigma,\delta}_{(o, \dots ,o)} ({\bf{a_1}}, \dots , {\bf{a_k}})$ of order $d$ coincides with the skew Vandermonde matrix $V^{\sigma,\delta}_{d} (\{ {\bf{a_1}}, \dots , {\bf{a_k}}\})$ introduced in \cite[Definition 40]{evaluationandinterpolation}.
\end{remark}

\smallskip

The following result is a direct consequence of Theorem \ref{theorem 3.1.7} and it extends \cite[Corollary 42]{evaluationandinterpolation}.

\medskip

\begin{corollary}\label{Vandermonde}
Let $\Omega=\{\bf{a_{1}},\bf{a_2}, \dots ,\bf{a_k}\} \subseteq \mathbb{F}^{n}$ be a finite set and let $m_1,\dots,m_k\in \mathcal{M}$. Putting $N:=\sum_{i=1}^{k}(\deg(m_i)+1)$, the following conditions are equivalent:
\begin{itemize}
\item[$1)$] $\Omega$ is left (right) DP-independent of type $(m_1, \dots ,m_k)$ ;
 \item[$2)$] given any finite set of $N$ values in $\mathbb{F}$
$$\{b_{j,o},b_{j,x_{j_1}},b_{j,x_{j_2}x_{j_1}},\dots,b_{j,m_j}:j=1,2,\dots,k\}\ ,$$ 
there exists a skew polynomial $F\in \mathcal{R}$ with $\deg(F)< N$ such that 
$$\qquad F({\bf{a_j}})=b_{j,o}, \Delta_{{\bf{a_j}}}^{x_{j_1}}F({\bf{a_j}})=b_{j,x_{j_1}},\Delta_{{\bf{a_j}}}^{x_{j_2}x_{j_1}}F({\bf{a_j}})=b_{j,x_{j_2}x_{j_1}}, \dots , \Delta_{{\bf{a_j}}}^{m_j}F({\bf{a_j}})=b_{j,m_j}$$ 
$$\qquad\left(F_L({\bf{a_j}})=b_{j,o}, \Delta_{{\bf{a_j}},L}^{x_{j_1}}F({\bf{a_j}})=b_{j,x_{j_1}},\Delta_{{\bf{a_j}},L}^{x_{j_2}x_{j_1}}F({\bf{a_j}})=b_{j,x_{j_2}x_{j_1}}, \dots , \Delta_{{\bf{a_j}},L}^{m_j}F({\bf{a_j}})=b_{j,m_j}\right)$$
for all $j=1,\dots,k$ ;
\item[$3)$] writing $r:=\frac{n^N-1}{n-1}=1+n+\dots+n^{N-1}$, for any 
$${\bf{b_j}}:=(b_{j,o},b_{j,x_{j_1}},b_{j,x_{j_2}x_{j_1}},\dots,b_{j,m_j})\in\mathbb{F}^{\deg (m_j)+1}$$ with $j=1, \dots ,k$, there exists a solution $(F_1,\dots,F_r)\in\mathbb{F}^{r}\ \left(\ {}^t(F_1,\dots,F_r)\in\mathbb{F}^{r\times 1}\ \right)$ to the linear system 
$$(z_1,\dots,z_r)\cdot V^{\sigma,\delta}_{(m_1, \dots ,m_k)} ({\bf{a_1}}, \dots , {\bf{a_k}})=(\bf{b_{1}},\dots ,\bf{b_k})$$
$$\left(V^{\sigma,\delta}_{(m_1, \dots ,m_k),L} ({\bf{a_1}}, \dots , {\bf{a_k}})\cdot {}^{t}(z_1,\dots,z_r)=
\left(
\begin{matrix}
{}^{t}\bf{b_{1}} \\
\vdots \\
{}^{t}\bf{b_k}
\end{matrix}
\right)\ \right) ,$$
where $(z_1,\dots,z_r)\ ({}^t(z_1,\dots,z_r))$ is a vector of variables in $\mathbb{F}^r \left(\ \mathbb{F}^{r\times 1}\ \right)$ and 
$$V^{\sigma,\delta}_{(m_1, \dots ,m_k)} ({\bf{a_1}}, \dots , {\bf{a_k}})\ \left(V^{\sigma,\delta}_{(m_1, \dots ,m_k),L} ({\bf{a_1}}, \dots , {\bf{a_k}})\right)$$
is the $r\times N\ \left(N\times r\right)$ right (left) $(\sigma,\delta)$-confluent Vandermonde matrix of order $N$ on the set $\{ {\bf{a_1}}, \dots , {\bf{a_k}}\}$.
\end{itemize}
\end{corollary} 

\begin{remark}
In Corollary \ref{Vandermonde}, the right (left) $(\sigma,\delta)$-confluent Vandermonde matrix of order $N$ on the set $\{ {\bf{a_1}}, \dots , {\bf{a_k}}\}$ is of square form if and only if $n^N=(n-1)N+1$. This numerical condition is alway satisfied for all $N\in\mathbb{Z}_{\geq 1}$ when $n=1$, while for any fixed $n\geq 2$, it holds only for $N=0,1$. This shows that for $n\geq 2$ the matrix $V^{\sigma,\delta}_{(m_1, \dots ,m_k)} ({\bf{a_1}}, \dots , {\bf{a_k}})\ \left(V^{\sigma,\delta}_{(m_1, \dots ,m_k),L} ({\bf{a_1}}, \dots , {\bf{a_k}})\right)$ is not of square form, except for $N=1$. This fact implies that for $N>1$ one can only ensure the existence of skew interpolating polynomials, and not the uniqueness, except for the univariate case (i.e. when $n=1$).  
\end{remark}

\begin{remark}
In general, the skew Vandermonde matrices are useful tools to explicitly construct Lagrange (Hermite-type) skew interpolating polynomials, although the uniqueness of these polynomials cannot generally be guaranteed when $n\geq 2$ due to the non-square form of these matrices, as shown in the previous remark. More precisely, from Corollary \ref{Vandermonde}, it follows that if $(F_1,\dots,F_r)\in\mathbb{F}^{r}\ \left(\ {}^t(F_1,\dots,F_r)\in\mathbb{F}^{r\times 1}\ \right)$ is a solution to the linear system 
$$(z_1,z_2,\dots,z_r)\cdot 
(V_{m_1}^{\sigma,\delta}({\bf{a_1}})\ |\ V_{m_2}^{\sigma,\delta} ({\bf{a_2}})\ |\ \dots \ |\ V_{m_k}^{\sigma,\delta}({\bf{a_k}}))
=(\bf{b_{1}},\dots ,\bf{b_k})$$
$$\left(\
\left(
\begin{matrix}
V_{m_1,L}^{\sigma,\delta}({\bf{a_1}})\\ 
V_{m_2,L}^{\sigma,\delta} ({\bf{a_2}})\\ 
\vdots \\ 
V_{m_k,L}^{\sigma,\delta}({\bf{a_k}})
\end{matrix}
\right)
\cdot 
\left(
\begin{matrix}
z_1 \\
z_2 \\
\vdots \\
z_r
\end{matrix}
\right)
=
\left(
\begin{matrix}
{}^{t}\bf{b_{1}} \\
\vdots \\
{}^{t}\bf{b_k}
\end{matrix}
\right)\ \right) ,$$ 
where $\{ m'_1,\dots,m'_r\}$ is an ordered set of all distinct monomials $m'_i$ in $\mathcal{M}$ of degree less than $\sum_{i=1}^{k}(\deg(m_i)+1)$ as in Definition \ref{DefVandermonde}, then a skew interpolating polynomial $F\in \mathcal{R}$ with $\deg(F)< \sum_{i=1}^{k}(\deg(m_i)+1)$ and such that 
$$\qquad F({\bf{a_j}})=b_{j,o}, \Delta_{{\bf{a_j}}}^{x_{j_1}}F({\bf{a_j}})=b_{j,x_{j_1}},\Delta_{{\bf{a_j}}}^{x_{j_2}x_{j_1}}F({\bf{a_j}})=b_{j,x_{j_2}x_{j_1}}, \dots , \Delta_{{\bf{a_j}}}^{m_j}F({\bf{a_j}})=b_{j,m_j}$$ 
$$\qquad\left(F_L({\bf{a_j}})=b_{j,o}, \Delta_{{\bf{a_j}},L}^{x_{j_1}}F({\bf{a_j}})=b_{j,x_{j_1}},\Delta_{{\bf{a_j}},L}^{x_{j_2}x_{j_1}}F({\bf{a_j}})=b_{j,x_{j_2}x_{j_1}}, \dots , \Delta_{{\bf{a_j}},L}^{m_j}F({\bf{a_j}})=b_{j,m_j}\right)$$
for all $j=1,\dots,k$, is given by
$F=\sum_{i=1}^{r}F_{i}\ m'_{i} \ \ \ \left(\ F=\sum_{i=1}^{r}m'_{i}\ {F}_{i}\ \right)\ .$
\end{remark}

\bigskip

\noindent {\bf Acknowledgement.} This work was partially supported by the Proyecto VRID N. 219.015.023-INV of the University of Concepci\'on.

\end{document}